\documentclass[a4paper,11pt]{amsart}

\usepackage[utf8]{inputenc}
\usepackage[T1]{fontenc}
\usepackage{amscd,amssymb,amsmath,latexsym,graphicx,epsfig,color,url, tikz}
\usepackage[active]{srcltx}
\usepackage[all]{xy}
\usepackage{hyperref}
\hypersetup{breaklinks=true}

\usepackage{todonotes}

\newcommand{\conv}{\operatorname{conv}}

\newcommand{\Res}{\operatorname{Res}}
\newcommand{\Elim}{\operatorname{Elim}}

\newcommand{\Det}{\operatorname{Det}}

\newcommand{\A}{\mathbb{A}}
\newcommand{\C}{\mathbb{C}}

\newcommand{\K}{\mathbb{K}}

\newcommand{\N}{\mathbb{N}}
\renewcommand{\P}{\mathbb{P}}
\newcommand{\Q}{\mathbb{Q}}
\newcommand{\R}{\mathbb{R}}
\newcommand{\T}{\mathbb{T}}

\newcommand{\Z}{\mathbb{Z}}

\newcommand{\cA}{{\mathcal A}}

\newcommand{\cF}{{\mathcal F}}

\newcommand{\cO}{{\mathcal O}}

\newcommand{\cV}{{\mathcal V}}

\newcommand{\bft}{{\boldsymbol{t}}}

\newcommand{\bfxi}{{\boldsymbol{\xi}}}
\newcommand{\bfchi}{{\boldsymbol{\chi}}}

\newcounter{thm}
\numberwithin{thm}{section}
\numberwithin{equation}{section}

\theoremstyle{definition}
\newtheorem{definition}[thm]{Definition}

\newtheorem{remark}[thm]{Remark}
\newtheorem{example}[thm]{Example}

\theoremstyle{plain}
\newtheorem{lemma}[thm]{Lemma}
\newtheorem{proposition}[thm]{Proposition}
\newtheorem{theorem}[thm]{Theorem}
\newtheorem{corollary}[thm]{Corollary}

\begin{document}

\title{Sparse Nullstellensatz, resultants and determinants of complexes}

\author{Carlos D'Andrea}
\address{Departament de Matem\`atiques i Inform\`atica, Universitat de Barcelona. Gran
Via 585, 08007 Barcelona, Spain \&  Centre de Recerca Matem\`atica, Edifici C, Campus Bellaterra, 08193 Bellaterra, Spain
}
\email{cdandrea@ub.edu}
\urladdr{http://www.ub.edu/arcades/cdandrea.html}

\author{Gabriela Jeronimo}
\address{Departamento de Matem\'atica, FCEN, Universidad de Buenos Aires, Ciudad Universitaria - Pab. I, (1428) Buenos Aires, Argentina, and
CONICET-Universidad de Buenos Aires, Instituto de Investigaciones Matem\'aticas ''Luis A. Santal\'o'' (IMAS), Buenos Aires, Argentina}
\email{jeronimo@dm.uba.ar}
\urladdr{http://mate.dm.uba.ar/~jeronimo}

\begin{abstract}
We refine and extend a result by Tuitman on the supports of a B\'ezout identity satisfied by a finite sequence of sparse Laurent polynomials without common zeroes in the toric variety associated to their supports. When the number of these polynomials is one more than the dimension of the ambient space, we obtain a formula for computing the sparse resultant as the determinant of a Koszul type complex.
\end{abstract}


\maketitle

\section{Introduction and statement of the main results}
Given a field $K,$  and a finite sequence $f_1,\ldots, f_k$ in $R_K:=K[t_1,\ldots, t_n],$   the classical {\em Hilbert's Nullstellensatz} (\cite{hil93}) states that the ideal generated by these elements is the whole ring $R_K$ if and only if the variety defined by them in the affine space $\A^n_{\overline{K}}$  over $\overline{K},$ the algebraic closure of $K,$ is empty.  The {\em effective Nullstellensatz } deals with the problem of giving an algorithmic procedure to decide if this is the case, and if so to compute a {\it B\'ezout identity}
\begin{equation}\label{bi}
\sum_{i=1}^k g_if_i=1
\end{equation}
from the input sequence $f_1,\ldots, f_k.$  Note that if degree bounds for the $g_i$'s are provided, then \eqref{bi} turns into an easy to verify Linear Algebra routine. This was the original approach of Grete Hermann in \cite{her26}, and a lot of improvement has been done since then, see for instance \cite{DKS13} and the references therein.

A first attempt to deal with this problem can be made by homogenizing the input, i.e.~considering $F_1,\ldots, F_k\in \tilde{R}_K:=K[t_0,\ldots,t_n]$, which are the homogenizations of $f_1,\ldots, f_k$ up to their total degrees $d_1,\ldots, d_k$ respectively. We can then  study the Koszul complex associated to this sequence:
\begin{equation}\label{occis}
Kos(\tilde{R}_K,F_1,\ldots, F_k):\ 0\to\ \tilde{R}_K^{k\choose k}\stackrel{\Psi_k}{\to} \tilde{R}_K^{k\choose k-1}\stackrel{\Psi_{k-1}}{\to}\ldots\stackrel{\Psi_2}{\to} \tilde{R}_K^{k\choose1}\stackrel{\Psi_1}{\to} \tilde{R}_K^{k\choose 0}\to 0,
\end{equation}
 where the maps $\Psi_i$ are defined by setting:
\begin{equation}\label{axi}
\Psi_i(e_{j_1,\ldots, j_i}):=\sum_{\ell=1}^i (-1)^{\ell+1} F_{j_\ell}\, e_{j_1,\ldots, \check{j_\ell},\ldots, j_i},
\end{equation}
and we denote the standard basis of the free $\tilde{R}_K$-module $ \tilde{R}_K^{k\choose i}$ with $\big(e_{j_1,\ldots, j_i}\big)_{1\leq j_1<\ldots <j_i\leq k}.$
Note that $\Psi_1=(F_1,\ldots, F_k).$
As this sequence is homogeneous, the maps in the complex \eqref{occis} are also graded as well, and hence we can consider the different graded pieces of \eqref{occis}. For a non-negative integer $d,$ the  $d$-th graded part of $Kos(\tilde{R}_K,F_1,\ldots, F_k)$ (the one which has degree $d$ in $\tilde{R}_K^{k\choose 0}$ above) will be denoted by $Kos(\tilde{R}_K,F_1,\ldots, F_k)_d.$
The following result is well-known, see Theorem $5.1$ in \cite{DD01} for the case $k=n+1.$ The general case can be reduced to this one by using standard properties of Koszul complexes of sets containing regular sequences (cf. \cite[Corollary 1.6.13]{BH93}).
\begin{theorem}\label{g0}
Let  $d\geq d_1+\ldots+d_k-n.$ The following are equivalent:
\begin{enumerate}
\item there are no common zeroes of $F_1,\ldots, F_k$ in projective space $\P^n_{\overline{K}};$
\item $Kos(\tilde{R}_K,F_1,\ldots, F_k)_d$ is exact as a complex of $K$-vector spaces;
\item the rightmost nontrivial map of $Kos(\tilde{R}_K,F_1,\ldots, F_k)_d$ is onto.
\end{enumerate}
\end{theorem}

Theorem \ref{g0} provides an optimal bound (if $d< d_1+\ldots+d_k-n$ the result does not hold for $k=n+1$, see for instance Theorem $1.5.16$ in \cite{CD05a}) on the degrees of the $g_i$'s which implies a B\'ezout identity as in \eqref{bi} at the cost of requiring the input not to have zeroes not only in affine space but in a compactification of it, the $n$-th dimensional projective space $\P^n_{\overline{K}}.$  Indeed from the exactness of  $Kos(\tilde{R}_K,F_1,\ldots, F_k)_d$ one can find homogeneous polynomials $G_1,\ldots, G_k\in \tilde{R}_K$ such that
$\sum_i G_iF_i = t_0^d.$ By setting $t_0=1$, we recover $g_1,\ldots, g_k$ such that $\deg(g_i)\leq d_1+\ldots+d_k-n-d_i,$ satisfying \eqref{bi}, and hence solving the effective Nullstellensatz problem.
\begin{remark}
As we are dealing with homogeneous polynomials and their zeroes in projective space, for the sequence
$F_1,\ldots, F_k$ to have no common zeroes, by Krull's Hauptidealsatz (cf. Theorem $10.2$ in \cite{eis95}) we are necessarily forced to have $k\geq n+1.$
\end{remark}

\begin{remark}
The inequality $d\geq d_1+\ldots+d_k-n$ is also sharp, as it is also well-known (see for instance Theorem $1$ in Section $5.11$ of \cite{AGV85}) that if $k=n+1,$ the Jacobian of $F_1,\ldots, F_{n+1}$ has degree $d_1+\ldots+d_{n+1}-n-1,$ and does not belong to the ideal generated by the input polynomials if they define the empty variety in $\P^n_{\overline{K}},$ hence the rightmost nontrivial map of this graded piece of $Kos(\tilde{R}_K,F_1,\ldots, F_{n+1})$ cannot be onto.
\end{remark}

\begin{remark}\label{naiv}
The strongest hypothesis of requiring $F_1,\ldots, F_k$ not having common zeroes in $\P^n_{\overline{K}}$ actually forces very low bounds for the degrees of the $g_i$'s in the Nullstellensatz. In general it is expected that  $\deg(g_j)\leq \prod_{i=1}^k d_i-d_j,$ see for instance Theorem $1$ in \cite{DKS13}, and this bound is sharp. This shows that the approach to the Nullstellensatz via homogenization is far from being enough and sharp, and extra work should be done to deal with this problem in general.
\end{remark}

Of special interest is the so-called ``codimension one'' case, i.e. when $k=n+1.$ Here,  tools from Elimination Theory like resultants can be used to deal with this problem. To introduce them, consider the following generic homogeneous polynomials in $n+1$ variables $t_0,\ldots, t_n$ of respective positive degrees $d_1,\ldots, d_{n+1}:$

\begin{equation}\label{fig}
F^g_i:=\sum_{\alpha\in\Z_{\geq0}^{n+1},\,|\alpha|=d_i} c_{i,\alpha} t^\alpha, \ i=1,\ldots, n+1,
\end{equation}
with each of the $c_{i,\alpha}$ being a new indeterminate. Set $\K:=\Q(c_{i,\alpha}).$  We easily note that the sequence $F^g_1,\ldots, F^g_{n+1}$ does not have common zeroes in $\P^n_{\overline\K}.$ Hence, the complex  $Kos(\tilde{R}_\K,F^g_1,\ldots, F^g_{n+1})_d$  is exact for $d\geq d_1+\ldots+d_{n+1}-n$ thanks to Theorem \ref{g0}. In the special case $n=1,\,d=d_1+d_2-1,$ \eqref{occis} reduces to only one nontrivial map between $\K$-vector spaces of the same dimension $d_1+d_2,$ whose determinant equals the classical Sylvester resultant of the homogeneous bivariate forms $F_1$ and $F_2,$ see \cite[Chapter $3$ \S 1]{CLO05}.  This resultant can be generalized for $n>1:$ given positive degrees $d_1,\ldots, d_{n+1}$, there is an irreducible polynomial called $\Res_{d_1,\ldots, d_{n+1}}\in\Z[c_{i,\alpha}]$ called the resultant of $F^g_1,\ldots, F^g_{n+1},$ which vanishes after a specialization of the $c_{i,\alpha}$ in a field $K$  if and only if the specialized system arising from \eqref{fig} after this specialization has a solution in the projective space over $\overline{K},$ see \cite[Chapter $3$ \S 2]{CLO05} for its proper definition and basic properties.

In another direction, the determinant of a linear map among spaces of the same dimension with fixed bases can be generalized to the determinant of a finite exact complex of $K$-vector spaces of finite dimension, fixing bases in each of the nonzero terms of it. In practice, it involves the computation of maximal minors of all the nontrivial matrices representing the differentials in the aforementioned bases.  An introduction to this topic can be found in \cite{del87} but this concept goes also all the way back to Cayley, see for instance \cite[Appendix A]{GKZ94}. Computational aspects can be found in \cite{dem84,cha93}, we refer the reader to Section \ref{detc} for more.

The following is a classical result which generalizes the computation of the Sylvester resultant of two univariate polynomials by computing the determinant of a matrix. See \cite{cha93} and also Theorem $5.1$ in \cite{DD01} for proofs. In our case, all the bases that we are going to consider to compute determinants are the monomial bases of each of the terms in the complex.

\begin{theorem}\label{chom}
For  $d\geq d_1+\ldots+d_{n+1}-n,$ we have that:
\begin{enumerate}
\item Up to a sign,  the determinant of the complex $Kos(\tilde{R}_\K,F^g_1,\ldots, F^g_{n+1})_d$ is equal to $\Res_{d_1,\ldots, d_{n+1}}.$
\item This determinant can be computed as the $\gcd$ of the maximal minors of the $d$-th graded piece of $\Psi_{1},$  the rightmost nontrivial map in $Kos(\tilde{R}_\K,F^g_1,\ldots, F^g_{n+1})_d$.
\item Let $K$ be any field. For homogeneous polynomials $F_1,\ldots, F_{n+1}\in K[t_0,\ldots, t_n]$ of respective degrees $d_1,\ldots, d_{n+1},$ the complex $Kos(\tilde{R}_K,F_1,\ldots, F_{n+1})_d$ is exact if and only if $\Res_{d_1,\ldots, d_{n+1}}(F_1,\ldots, F_{n+1})\neq0.$
\end{enumerate}
\end{theorem}

In this paper, we are going to present extensions of Theorems \ref{g0} and \ref{chom} to the sparse setting.
To do this, we start with finite sets $\cA_1,\ldots, \cA_{k}\subset\Z^n.$  For a field $K$ of characteristc zero, and  $i=1,\ldots, k,$ consider polynomials
\begin{equation}\label{fi}
f_i:=\sum_{b\in\cA_i} f_{i,b} t^b \in K[t_1^{\pm1},\ldots, t_n^{\pm1}],
\end{equation}
with $f_{i,b}\in K^*:=K\setminus\{0\}.$ We say that $\cA_i$ is the {\em support} of $f_i, \ i=1,\ldots, k.$ For short we denote with $K[t^{\pm1}]$ the ring of Laurent polynomials $K[t_1^{\pm1},\ldots, t_n^{\pm1}].$ The torus $(\overline{K}^*)^n$ will be denoted with $\T_{\overline{K}}^n.$ It is the natural environment to look for common zeroes of  \eqref{fi}.

Set $P_i:=\mbox{conv}(\cA_i)\subset\R^n$ for the convex hull of $\cA_i,\, i=1,\ldots, k.$  It is an {\em integral} polytope, meaning that all its vertices are in $\Z^n.$ We further set
\begin{equation}\label{pp}
P:=P_1+\ldots + P_k,
\end{equation}
where the operation in the right hand side is the Minkowski sum of polytopes, as defined in \cite[Chapter 7]{CLO05}.
If $P$ is $n$-dimensional, there exists a \textit{toric variety} $X_P$ containing $\T_{\overline{K}}^n$ as a dense subset, and a polynomial ring $S$ having as many variables as facets of $P,$ and such that  $f_1,\ldots, f_k$ get ``homogenized'' to $F_1,\ldots, F_k$, elements of a polynomial ring $S$ to be described later.  Homogeneous here means with respect to the grading given by the Chow group of $X_P$  (see \eqref{eFFe} to get the homogenization formula) which satisfies that $\xi\in\T_{\overline{K}}^n$ is a common zero of the $f_i$'s if and only if it is  a common zero of the $F_i$'s.

To illustrate our results and help the reader digest the technical background, we consider the following example which we will carry out all along the paper.
\begin{example}\label{running}
{\em 
 Set $n=2$, and $k=3.$ We will simplify the subindexes to avoid an overuse of notation. Let 
$$
f_1:=f_{10}+f_{11} t_1t_2, \ \ f_2:=f_{20}+f_{21}t_1^2t_2+f_{22}t_1t_2^2, \ \ f_3:=f_{30}+f_{31}t_1^2t_2+f_{32}t_2.
$$
In this case, we have
$$\cA_1=\{(0,0),\, (1,1)\},\ \  \cA_2=\{(0,0),\,(2,1),\,(1,2)\}, \ \ \cA_3=\{(0,0),\, (2,1),\,(0,1)\},
$$
and their convex hulls $P_1,\,P_2$ and $P_3$ are a segment and two triangles in $\R^2,$ respectively.  The Minkowski sum $P=P_1+P_2+P_3$ is the heptagon in $\R^2$ with vertices \newline $(0,0),\,(0,1),\,(1,3),\,(2,4),\,(4,4),\,(5,3),$ and $(4,2)$, see Figure \ref{fig:supports}.
}

\end{example}

\begin{figure}[ht]
\begin{tikzpicture}
\draw[->, color= darkgray] (-0.5,0) -- (2.5,0);
\draw[->, color= darkgray] (0,-0.5) -- (0,2.5);
 \filldraw[color=red] (0,0) circle (2pt);
 \filldraw[color=red] (1,1) circle (2pt);
 \draw[thick] (0,0)--(1,1);
 \draw[color=red] (0.5,-0.3) node {$\cA_1$};
 \draw (0.3,0.5) node[above] {$P_1$};
\end{tikzpicture} \hspace{5mm}
\begin{tikzpicture}
\draw[->, color= darkgray] (-0.5,0) -- (2.5,0);
\draw[->, color= darkgray] (0,-0.5) -- (0,2.5);
 \filldraw[color=green!50!black] (0,0) circle (2pt);
 \filldraw[color=green!50!black] (2,1) circle (2pt);
 \filldraw[color=green!50!black] (1,2) circle (2pt);
 \filldraw[color=gray!30] (0,0)--(2,1)--(1,2)--(0,0);
 \draw[very thick] (0,0)--(2,1)--(1,2)--(0,0);
  \draw[color=green!50!black] (1,-0.3) node {$\cA_2$}; 
 \draw (1.5,1.5) node[above right] {$P_2$};
\end{tikzpicture} \hspace{5mm}
\begin{tikzpicture}
\draw[->, color= darkgray] (-0.5,0) -- (2.5,0);
\draw[->, color= darkgray] (0,-0.5) -- (0,2.5);
 \filldraw[color=green!50!black] (0,0) circle (2pt);
 \filldraw[color=green!50!black] (2,1) circle (2pt);
 \filldraw[color=green!50!black] (0,1) circle (2pt);
 \filldraw[color=gray!30] (0,0)--(2,1)--(0,1)--(0,0);
 \draw[very thick] (0,0)--(2,1)--(0,1)--(0,0);
  \draw[color=green!50!black] (1,-0.3) node {$\cA_3$};
 \draw (1,1.5) node {$P_3$};
\end{tikzpicture}

\bigskip
\begin{tikzpicture}
\draw[->, color= darkgray] (-0.5,0) -- (5.5,0);
\draw[->, color= darkgray] (0,-0.5) -- (0,4.5);
\filldraw (0,0) circle (2pt);
\filldraw (0,1) circle (2pt);
\filldraw (1,3) circle (2pt);
\filldraw (2,4) circle (2pt);
\filldraw (4,4) circle (2pt);
\filldraw (5,3) circle (2pt);
\filldraw (4,2) circle (2pt);
\filldraw[color=gray!30] (0,0)--(0,1)--(1,3)--(2,4)--(4,4)--(5,3)--(4,2)--(0,0);
\draw[very thick] (0,0)--(0,1)--(1,3)--(2,4)--(4,4)--(5,3)--(4,2)--(0,0);
\draw (3,0.5) node {$P=P_1+P_2+P_3$};
\end{tikzpicture}

   \caption{Supports, convex hulls and Miknowski sum of the polynomials in Example  \ref{running}.} \label{fig:supports}
\end{figure}

To deal with a more general situation, assume that we are given another integral polytope $Q\subset\R^n.$ Associated to the Minkowski sum $P+Q$ there is a toric variety $X_{P+Q}$ ``lying over'' $X_P,$ with its corresponding homogeneous coordinate ring $S'$ and hence $f_1,\ldots, f_k$ can also be homogenized to $F_1',\ldots, F_k'\in S'.$ When $Q$ is a single point, we recover the variety $X_P$ so this case covers the previous one and extends to a more general setup.

For each of the facets of $P+Q$ there is an invariant Weil divisor associated to it. We will denote by $D_1,\ldots, D_{n+r}$ these divisors.
As in Canny and Emiris' matrix construction for the resultant (cf.~\cite{CE93}), we fix a vector $\delta\in\Q^n.$ Associated to it we define the divisor $D_\delta:=\sum_j -D_j$, the sum being over all the divisors associated to facets whose inner normal $\eta_j$ satisfies $\langle \eta_j,\delta\rangle>0.$ If there are no such facets, then $D_\delta$ is set equal to zero.
Let $\alpha_\delta$ be the degree of $D_\delta.$   Each of the $F'_i$  has a degree $\alpha_{P_i}$ associated to its Newton polytope $P_i.$
Set $\alpha_P=\sum_{i=1}^k \alpha_{P_i},$ and $\alpha_Q$ the degree associated to the polytope $Q$ in $X_{P+Q}.$

As in the homogeneous case, we can consider the Koszul complex $Kos(S',F_1',\ldots, F_k'),$ and its graded pieces determined by the Chow group of $X_{P+Q}.$
The following is our generalization of Theorem \ref{g0} to the sparse setting:
\begin{theorem}\label{g1}
Let $K$ be a field of characteristic zero,  $\cA_1,\ldots, \cA_{k}\subset\Z^n$ finite sets, $f_1,\ldots, f_k$ as in \eqref{fi}, and $P$ as in \eqref{pp}. If $P$ is an $n$-dimensional polytope,  the following are equivalent:
\begin{enumerate}
\item there are no common zeroes of $F_1,\ldots, F_k$ in $X_P,$ where $F_i$ is the homogenization of $f_i$ explained above and defined properly in \eqref{eFFe};
\item $Kos(S',F_1',\ldots, F_k')_{\alpha_\delta+\alpha_P+\alpha_Q}$ is exact as a complex of $K$-vector spaces;
\item the rightmost nontrivial map of $Kos(S',F_1',\ldots, F_k')_{\alpha_\delta+\alpha_P+\alpha_Q}$ is onto.
\end{enumerate}
\end{theorem}

\begin{remark}
The hypothesis of $P$ being $n$-dimensional can be dropped by the simple observation that if the dimension of this polytope is smaller, then one can reparametrize \eqref{fi} with fewer than $n$ variables.
\end{remark}

From Theorem \ref{g1} we derive the following sharp sparse Nullstellensatz for polynomials without zeroes in $X_P.$
\begin{theorem}\label{spr}
Let $K$ be a field of characteristic zero,  $\cA_1,\ldots, \cA_{k}\subset\Z^n$ finite sets, $f_1,\ldots, f_k$ as in \eqref{fi}, and $P$ as in \eqref{pp}. If $\delta\in\big((-1,1)\cap\Q\big)^n,$ and $P$ is an $n$-dimensional polytope, 
for any integral polytope $Q$, the following are equivalent:
\begin{enumerate}
\item there are no common zeroes of $F_1,\ldots, F_k$ in $X_P;$
\item for all $g\in K[t^{\pm1}]$ supported in $(P+Q+\delta)\cap\Z^n,$ there exist $g_1,\ldots, g_k\in K[t^{\pm1}]$ supported in $(P_1+\ldots +P_{i-1}+P_{i+1}+\ldots+P_k+Q+\delta)\cap\Z^n,\,1\leq i\leq k$ respectively, such that
\begin{equation}\label{coz}
g=g_1f_1+\ldots+g_kf_k.
\end{equation}
\end{enumerate}
\end{theorem}

Note that if $\delta\notin\Z^n,$  then for any integral polytope $R$ we have that $\#(R+\delta)\cap\Z^n<\#R\cap\Z^n,$ so the addition of $\delta$ above actually shrinks the supports in the Bezout identity \eqref{coz}.

There are several ``sparse Nullstellensatze without zeroes in $X_P$'' in the literature with a similar flavour than Theorem \ref{spr}.  To mention some,  Theorem $3$ in \cite{CDV06}  presents a Nullstellensatz in the same conditions for $k=n+1,$ all the $P_i$'s being equal to the same $n$-dimensional polytope, $Q=\{\bf0\},$ and $\delta=(0,\ldots, 0).$
With different motivation and techniques,  Theorem $1.2$ in \cite{wul11} deals with the case of polynomials in $\C[t_1,\ldots, t_n]$  having the proper codimension in the toric variety and no components at infinity. Their result holds for a ``large poltyope''  $P_0$  containing all the vectors of the standard basis of $\Z^n,$ and the support of each of the $f_i$'s. So, in general larger than our $P$.

The closest to our main result is Theorem $1.3$ in \cite{tui11}, where the claim is proven with our hypothesis for the case  $Q=\{\bf0\}$ and  $\delta=(0,\ldots, 0).$ Theorem \ref{spr} above represents an improvement over this result. For instance, in our running Example \ref{running},  the polynomials $g_1, \, g_2,\, g_3$ in \eqref{coz} have supports of cardinality $11,\, 7,\,7$  respectively if one uses Tuitman's results. 
From our results above, these cardinalities can shrink to $7, 5 ,4$ if $\delta=(\frac12,\frac12)$, and $8,\,4,\,4$  if $\delta=(0,-\frac12),$ see below.

Moreover, the complex  $Kos(S',F_1',\ldots, F_k')_{\alpha_\delta+\alpha_P+\alpha_Q}$ from Theorem \ref{g1} has the following shape for $\delta=(0, 0),$ and $Q=\{\bf0\}$ (this follows also  from Tuitman's, see Figure \ref{00}):
\begin{equation}\label{gg0}
0\to K \to K^2\oplus K^4\oplus K^4\to K^{11}\oplus K^7\oplus K^7\to K^{16}\to0.
\end{equation}

\begin{figure}[ht]\label{00}
\begin{tikzpicture}[scale=0.8]
\draw[->, color= darkgray] (-0.5,0) -- (6,0);
\draw[->, color= darkgray] (0,-0.5) -- (0,5);
\filldraw[color=gray!30] (0,0)--(0,1)--(1,3)--(2,4)--(4,4)--(5,3)--(4,2)--(0,0);
\draw[very thick] (0,0)--(0,1)--(1,3)--(2,4)--(4,4)--(5,3)--(4,2)--(0,0);
\draw (3.5,0.5) node {$P=P_1+P_2+P_3$};
\filldraw (0,0) circle (2pt);
\filldraw (0,1) circle (2pt);
\filldraw (1,3) circle (2pt);
\filldraw (2,4) circle (2pt);
\filldraw (4,4) circle (2pt);
\filldraw (5,3) circle (2pt);
\filldraw (4,2) circle (2pt);
\foreach \x in {2,3,4} {
            \foreach \y in {2,3,4} {
            \filldraw[color=black](\x,\y) circle (2pt);
            }}
\filldraw[color=black](1,1) circle (2pt);
\filldraw[color=black](1,2) circle (2pt);
\filldraw[color=black](2,1) circle (2pt);
\end{tikzpicture}

\bigskip
\begin{tikzpicture}[scale=0.8]
\draw[->, color= darkgray] (-0.5,0) -- (5,0);
\draw[->, color= darkgray] (0,-0.5) -- (0,4);
\filldraw[color=gray!30] (0,0)--(0,1)--(1,3)--(3,3)--(4,2)--(0,0);
\draw[very thick] (0,0)--(0,1)--(1,3)--(3,3)--(4,2)--(0,0);
\foreach \position in {(0,0),(0,1),(1,1), (1,2),(1,3),(2,1), (2,2),(2,3), (3,2), (3,3),(4,2)} 
\filldraw \position circle (2pt);
\draw (2,3.2) node[above] {$P_2+P_3$};
\end{tikzpicture}\hspace{5mm}
\begin{tikzpicture}[scale=0.8]
\draw[->, color= darkgray] (-0.5,0) -- (4,0);
\draw[->, color= darkgray] (0,-0.5) -- (0,4);
\filldraw[color=gray!30] (0,0)--(2,1)--(3,2)--(2,3)--(1,2)--(0,0);
\draw[very thick] (0,0)--(2,1)--(3,2)--(2,3)--(1,2)--(0,0);
\foreach \position in {(0,0),(1,1),(1,2),(2,1),(2,2),(2,3), (3,2)} 
\filldraw \position circle (2pt);
\draw (2,3.2) node[above] {$P_1+P_3$};
\end{tikzpicture}\hspace{5mm}
\begin{tikzpicture}[scale=0.8]
\draw[->, color= darkgray] (-0.5,0) -- (4,0);
\draw[->, color= darkgray] (0,-0.5) -- (0,4);
\filldraw[color=gray!30] (0,0)--(2,1)--(3,2)--(1,2)--(0,1)--(0,0);
\draw[very thick] (0,0)--(2,1)--(3,2)--(1,2)--(0,1)--(0,0);
\foreach \position in {(0,0),(0,1),(1,1),(1,2),(2,1),(2,2),(3,2)} 
\filldraw \position circle (2pt);
\draw (2,3.2) node {$P_1+P_2$};
\end{tikzpicture}
\end{figure}

Theorem \ref{g1} above shrinks the size of this complex considerably. For instance we have for $\delta=(\frac12,\frac12)$ and $Q=\{\bf0\},$ that  $Kos(S',F_1',\ldots, F_k')_{\alpha_\delta+\alpha_P+\alpha_Q}$ is isomorphic to the following complex of $K$-vector spaces,  see Figure \ref{fig00}:
\begin{equation}\label{gg1}
0\to K\oplus K^2\oplus K\to K^7\oplus K^5\oplus K^4\to K^{12}\to0,
\end{equation}

\begin{figure}[ht]
\begin{tikzpicture}[scale=0.8]
\draw[->, color= darkgray] (-0.5,0) -- (6,0);
\draw[->, color= darkgray] (0,-0.5) -- (0,5);
\filldraw[shift={(0.5, 0.5)}, color=gray!30] (0,0)--(0,1)--(1,3)--(2,4)--(4,4)--(5,3)--(4,2)--(0,0);
\draw[shift={(0.5, 0.5)}, very thick] (0,0)--(0,1)--(1,3)--(2,4)--(4,4)--(5,3)--(4,2)--(0,0);
\draw (3,0.5) node {$P+(\frac12, \frac12)$};
  \foreach \x in {2,3} {
            \foreach \y in {2,3,4} {
            \filldraw[color=black](\x,\y) circle (2pt);
            }}
\foreach \position in {(1,1), (1,2), (4,3), (4,4),(5,3),(5,4)} 
\filldraw \position circle (2pt);
\end{tikzpicture}

\medskip
\begin{tikzpicture}[scale=0.8]
\draw[->, color= darkgray] (-0.5,0) -- (5,0);
\draw[->, color= darkgray] (0,-0.5) -- (0,4);
\filldraw[shift={(0.5, 0.5)}, color=gray!30] (0,0)--(0,1)--(1,3)--(3,3)--(4,2)--(0,0);
\draw[shift={(0.5, 0.5)}, very thick] (0,0)--(0,1)--(1,3)--(3,3)--(4,2)--(0,0);
\foreach \position in {(1,1), (1,2),(2,2),(2,3), (3,2), (3,3),(4,3)} 
\filldraw \position circle (2pt);
\draw (1.5,0.5) node[right] {$P_2+P_3+(\frac12, \frac12)$};
\end{tikzpicture}\hspace{5mm}
\begin{tikzpicture}[scale=0.8]
\draw[->, color= darkgray] (-0.5,0) -- (4,0);
\draw[->, color= darkgray] (0,-0.5) -- (0,4);
\filldraw[shift={(0.5, 0.5)},color=gray!30] (0,0)--(2,1)--(3,2)--(2,3)--(1,2)--(0,0);
\draw[shift={(0.5, 0.5)},very thick] (0,0)--(2,1)--(3,2)--(2,3)--(1,2)--(0,0);
\foreach \position in {(1,1),(2,2),(2,3), (3,2),(3,3)} 
\filldraw \position circle (2pt);
\draw (1,0.5) node[right] {$P_1+P_3+(\frac12, \frac12)$};
\end{tikzpicture}\hspace{5mm}
\begin{tikzpicture}[scale=0.8]
\draw[->, color= darkgray] (-0.5,0) -- (4,0);
\draw[->, color= darkgray] (0,-0.5) -- (0,4);
\filldraw[shift={(0.5, 0.5)},color=gray!30] (0,0)--(2,1)--(3,2)--(1,2)--(0,1)--(0,0);
\draw[shift={(0.5, 0.5)},very thick] (0,0)--(2,1)--(3,2)--(1,2)--(0,1)--(0,0);
\foreach \position in {(1,1),(1,2),(2,2),(3,2)} 
\filldraw \position circle (2pt);
\draw (1,0.5) node[right] {$P_1+P_2+(\frac12,\frac12)$};
\end{tikzpicture}

\medskip
\begin{tikzpicture}[scale=0.8]
\draw[->, color= darkgray] (-0.5,0) -- (2.5,0);
\draw[->, color= darkgray] (0,-0.5) -- (0,2.5);
 \draw[shift={(0.5, 0.5)},thick] (0,0)--(1,1);
 \filldraw (1,1) circle (2pt);
 \draw (0.4,0.3) node[right] {$P_1+(\frac12, \frac12)$};
\end{tikzpicture} \hspace{5mm}
\begin{tikzpicture}[scale=0.8]
\draw[->, color= darkgray] (-0.5,0) -- (2.5,0);
\draw[->, color= darkgray] (0,-0.5) -- (0,2.5);
 \filldraw[shift={(0.5, 0.5)},color=gray!30] (0,0)--(2,1)--(1,2)--(0,0);
 \draw[shift={(0.5, 0.5)},very thick] (0,0)--(2,1)--(1,2)--(0,0);
\foreach \position in {(1,1),(2,2)} 
\filldraw \position circle (2pt);
\draw (0.5,0.35) node[right] {$P_2+(\frac12,\frac12)$};
\end{tikzpicture} \hspace{5mm}
\begin{tikzpicture}[scale=0.8]
\draw[->, color= darkgray] (-0.5,0) -- (3,0);
\draw[->, color= darkgray] (0,-0.5) -- (0,2.5);
 \filldraw[shift={(0.5, 0.5)},color=gray!30] (0,0)--(2,1)--(0,1)--(0,0);
 \draw[shift={(0.5, 0.5)},very thick] (0,0)--(2,1)--(0,1)--(0,0);
\foreach \position in {(1,1)} 
\filldraw \position circle (2pt);
 \draw (2,0.35) node {$P_3+(\frac12, \frac12)$};
\end{tikzpicture}
   \caption{Supports of the monomials appearing in Example  \ref{running} with $\delta=(\frac12,\frac12)$.}  \label{fig00}

\end{figure}
while for $\delta=(0,-\frac12),$ we obtain (see Figure \ref{fig01}):
\begin{equation}\label{gg2}
0\to K^2\oplus K^2\to K^8\oplus K^4\oplus K^4\to K^{12}\to0.
\end{equation}

In all these complexes \eqref{gg0},\, \eqref{gg1},\, and \eqref{gg2},  the maps are of ``Koszul type'' as defined in \eqref{axi}. For instance, the rightmost nontrivial map in the three of them corresponds to an application of the form $(g_1, g_2, g_3)_\mapsto g_1\,f_1+g_2\,f_2+g_3\,f_3.$

\begin{figure}[ht]
\begin{tikzpicture}[scale=0.8]
\filldraw[shift={(0, -0.5)}, color=gray!30] (0,0)--(0,1)--(1,3)--(2,4)--(4,4)--(5,3)--(4,2)--(0,0);
\draw[shift={(0, -0.5)}, very thick] (0,0)--(0,1)--(1,3)--(2,4)--(4,4)--(5,3)--(4,2)--(0,0);
\draw (4,0.5) node {$P+(0, -\frac12)$};
\draw[->, color= darkgray] (-0.5,0) -- (6,0);
\draw[->, color= darkgray] (0,-0.5) -- (0,5);
 \foreach \x in {2,3} {
            \foreach \y in {1,2,3} {
            \filldraw[color=black](\x,\y) circle (2pt);
            }}
\foreach \position in {(0,0),(1,0),(1,1), (1,2), (4,2),(4,3)} 
\filldraw \position circle (2pt);
\end{tikzpicture}

\medskip
\begin{tikzpicture}[scale=0.8]
\filldraw[shift={(0, -0.5)}, color=gray!30] (0,0)--(0,1)--(1,3)--(3,3)--(4,2)--(0,0);
\draw[shift={(0, -0.5)}, very thick] (0,0)--(0,1)--(1,3)--(3,3)--(4,2)--(0,0);
\draw[->, color= darkgray] (-0.5,0) -- (5,0);
\draw[->, color= darkgray] (0,-0.5) -- (0,4);
\foreach \position in {(0,0),(1,0), (1,1), (1,2),(2,1),(2,2), (3,1),(3,2)} 
\filldraw \position circle (2pt);
\draw (1,3) node[right] {$P_2+P_3+(0, -\frac12)$};
\end{tikzpicture}\hspace{5mm}
\begin{tikzpicture}[scale=0.8]
\filldraw[shift={(0,-0.5)},color=gray!30] (0,0)--(2,1)--(3,2)--(2,3)--(1,2)--(0,0);
\draw[shift={(0, -0.5)},very thick] (0,0)--(2,1)--(3,2)--(2,3)--(1,2)--(0,0);
\draw[->, color= darkgray] (-0.5,0) -- (4,0);
\draw[->, color= darkgray] (0,-0.5) -- (0,4);
\foreach \position in {(1,0),(1,1),(2,1),(2,2)} 
\filldraw \position circle (2pt);
\draw (0.2,3) node[right] {$P_1+P_3+(0,-\frac12)$};
\end{tikzpicture}\hspace{5mm}
\begin{tikzpicture}[scale=0.8]
\filldraw[shift={(0, -0.5)},color=gray!30] (0,0)--(2,1)--(3,2)--(1,2)--(0,1)--(0,0);
\draw[shift={(0, -0.5)},very thick] (0,0)--(2,1)--(3,2)--(1,2)--(0,1)--(0,0);
\draw[->, color= darkgray] (-0.5,0) -- (4,0);
\draw[->, color= darkgray] (0,-0.5) -- (0,4);
\foreach \position in {(0,0),(1,0),(1,1),(2,1)} 
\filldraw \position circle (2pt);
\draw (0.2,3) node[right] {$P_1+P_2+(0,-\frac12)$};
\end{tikzpicture}

\medskip
\begin{tikzpicture}[scale=0.8]
\draw[->, color= darkgray] (-0.5,0) -- (2.5,0);
\draw[->, color= darkgray] (0,-0.5) -- (0,2.5);
 \draw[shift={(0,-0.5)},thick] (0,0)--(1,1);
 \draw (0.1,1) node[right] {$P_1+(0,- \frac12)$};
\end{tikzpicture} \hspace{5mm}
\begin{tikzpicture}[scale=0.8]
 \filldraw[shift={(0, -0.5)},color=gray!30] (0,0)--(2,1)--(1,2)--(0,0);
 \draw[shift={(0, -0.5)},very thick] (0,0)--(2,1)--(1,2)--(0,0);
 \draw[->, color= darkgray] (-0.5,0) -- (2.5,0);
\draw[->, color= darkgray] (0,-0.5) -- (0,2.5);
\foreach \position in {(1,0),(1,1)} 
\filldraw \position circle (2pt);
\draw (0.2,1.8) node[right] {$P_2+(0,-\frac12)$};
\end{tikzpicture} \hspace{5mm}
\begin{tikzpicture}[scale=0.8]
 \filldraw[shift={(0,-0.5)},color=gray!30] (0,0)--(2,1)--(0,1)--(0,0);
 \draw[shift={(0, -0.5)},very thick] (0,0)--(2,1)--(0,1)--(0,0);
\draw[->, color= darkgray] (-0.5,0) -- (3,0);
\draw[->, color= darkgray] (0,-0.5) -- (0,2.5);
\foreach \position in {(0,0),(1,0)} 
\filldraw \position circle (2pt);
 \draw (1.5,1.8) node {$P_3+(0, -\frac12)$};
\end{tikzpicture}
   \caption{Supports of the monomials appearing in Example  \ref{running} with $\delta=(0,-\frac12)$.}  \label{fig01}
\end{figure}

Our results are sharp also in terms of the size of the set of exponents, as it is known that  if one actually replaces $D_\delta$ with $\sum_{j=1}^{n+r} -D_j,$ then the corresponding graded piece of the last map of $Kos(S',F_1',\ldots, F_k')_{\alpha_\delta+\alpha_P+\alpha_Q}$ is never onto except for very degenerate situations. This is known if $k=n+1$  and all the polytopes $P_i$ being $n-$th dimensional (cf.~\cite{CD05}), and can be confirmed also in our running Example \ref{running}, as by counting the number of lattice points in the interior of the polytopes produces a complex  of the form
$$ 0\to K\to K^4\oplus K^2\oplus K\to K^7\to0,
$$
which can never be exact as the alternating sum of the dimensions above is not zero.

It should be emphasized, however, that general bounds for the sparse Nullstellensatz are known (see for instance \cite{som99}) and they are larger in size and nature than those presented in Theorem \ref{spr}. As in Remark \ref{naiv}, imposing stronger conditions on the input system is an advantage, but extra ideas are needed to deal with the general sparse Nullstellensatz than those treated here.

\smallskip
To generalize Theorem \ref{chom}, we need to consider generic families of $n+1$ polynomials with supports in $\cA_1,\ldots, \cA_{n+1}$ respectively:
\begin{equation}\label{gfi}
f^g_i:=\sum_{b\in\cA_i} c_{i,b} t^b, \ 1\leq i\leq n+1.
\end{equation}
with $c_{i,b}$ being new indeterminates.  As before, here we denote with $\K$ the field $\Q(c_{i,b}),$ and with $S'_\K:=\K[x_1,\ldots, x_{n+r}],$ the homogeneous coordinate ring of $X_{P+Q}$.

Associated with these data, there is a {\em sparse resultant} (\cite{GKZ94,CLO05, DS15}) which we will denote with $\Res_{\cA_1,\ldots, \cA_{n+1}}\in\Z[c_{i,b}],$ and it will be properly defined in Section \ref{sr}. The main property of this polynomial is that if it is not identically $1,$ then it vanishes under a specialization of $f^g_i\mapsto f_i\in K[t^{\pm1}]$ if and only if the homogenized system $F_1,\ldots, F_{n+1}\in S$ made from $f_1,\ldots, f_{n+1}$ with respect to $P_1,\ldots, P_{n+1}$ respectively as in  \eqref{eFFe}  has a nontrivial solution in $X_P.$

\begin{example}
{\em 
For our running Example \ref{running}, we have that, up to a sign, 
\begin{equation}\label{res123}
\begin{array}{ccl}
\Res_{\cA_1,\cA_2,\cA_3}&=&-f_{10} f_{11}^3 f_{21} f_{22} f_{30}^2 + f_{10} f_{11}^3 f_{20} f_{22} f_{30} f_{31} + 
  f_{10}^4 f_{22}^2 f_{31}^2 - f_{11}^4 f_{20} f_{21} f_{30} f_{32} \\ &&  + f_{11}^4 f_{20}^2 f_{31} f_{32} + 
  2 f_{10}^3 f_{11} f_{21} f_{22} f_{31} f_{32} + f_{10}^2 f_{11}^2 f_{21}^2 f_{32}^2.
  \end{array}
  \end{equation}}
  \end{example}
  
To avoid an excess of notation, we denote with $F^g_1,\ldots, F^g_{n+1}\in S'_\K$ the homogeneizations of \eqref{gfi} to the coordinate ring of $X_{P+Q}.$

For any degree $\beta,$ we have that $Kos(S'_\K,F^g_1,\ldots, F^g_{n+1})_{\beta}$ is a complex of finite dimensional $\K$-vector spaces, each of them either being trivial or having a monomial basis. When this complex is exact, its determinant  will be taken with respect to these monomial bases.

\begin{theorem}\label{csparse}
Let $K$ be a field of characteristic zero,  $\cA_1,\ldots, \cA_{k}\subset\Z^n$ finite sets,  $P$ as in \eqref{pp},  $f^g_1,\ldots, f^g_{n+1}$ be as in \eqref{gfi},  and $F^g_1,\ldots, F^g_{n+1}$ their respective $P_i$-homogenizations, $1\leq i\leq n+1.$ If $P$ is $n$-dimensional, $Q$ any integral polytope, and $\delta\in\Q^n,$ then  $Kos(S'_\K,F^g_1,\ldots, F^g_{n+1})_{\alpha_\delta+\alpha_P+\alpha_Q}$ is exact as a complex of $\K$-vector spaces. Moreover:
\begin{enumerate}
\item Up to a nonzero factor in $\Q,$ the determinant of $Kos(S'_\K,F^g_1,\ldots, F^g_{n+1})_{\alpha_\delta+\alpha_P+\alpha_Q}$ is equal to $\Res_{\cA_1,\ldots, \cA_{n+1}}.$
\item Up to a sign, $\Res_{\cA_1,\ldots, \cA_{n+1}}$ is equal to the $\gcd$ in $\Z[c_{i,b}]$ of the maximal minors of  the rightmost nontrivial map of $Kos(S'_\K,F^g_1,\ldots, F^g_{n+1})_{\alpha_\delta+\alpha_P+\alpha_Q}$.
\item Let $K$ be any field. For sparse polynomials $f_1,\ldots, f_{n+1}\in K[t^{\pm1}]$ with supports  $\cA_1,\ldots, \cA_{n+1}$ respectively, such that $\Res_{\cA_1,\ldots, \cA_{n+1}}\neq1,$ denote with $F'_1,\ldots, F'_{n+1}$ the homogeneizations of these polynomials as in \eqref{eFFe}.
 Then,   the complex $Kos(S'_K,F'_1,\ldots, F'_{n+1})_{\alpha_\delta+\alpha_P+\alpha_Q}$ is exact if and only if $\Res_{\cA_1,\ldots, \cA_{n+1}}(f_1,\ldots, f_{n+1})\neq0.$
\end{enumerate}
\end{theorem}

The degree $\alpha_\delta$ in the grading of the Koszul complex in Theorem \ref{spr} generalises the ``$-n$'' in the bound $d_1+\ldots+d_{n+1}-n$ appearing in Theorem \ref{chom}. In addition, all Canny-Emiris type matrices (\cite{CE93,DJS23}) for the resultant can be regarded as non singular maximal square submatrices of the rightmost nontrivial map in $Kos(S'_\K,F^g_1,\ldots, F^g_{n+1})_{\alpha_\delta+\alpha_P}$, for a suitable generic  $\delta\in\Q^n.$  The effective degree $\alpha_Q$ added corresponds to another polytope $Q,$  which also produces a new family of generalized Canny-Emiris type matrices as it was described in \cite[Remark $4.28$]{DJS23}.

\begin{example}
{\em Continuing with our running Example \ref{running}, we have that the determinants of the complexes \eqref{gg1} and \eqref{gg2} are equal to $\Res_{\cA_1,\cA_2,\cA_3}$ described in \eqref{res123}. It is well-known  that nonzero maximal minors for the rightmost nontrivial map of these two complexes can be obtained ``a la Canny-Emiris'' as explained in \cite{CE93, DJS23}. From Theorem \ref{csparse} we now deduce that $\Res_{\cA_1,\cA_2,\cA_3}$ is the $\gcd$ of all these maximal minors. 
}
\end{example}

It should be also mentioned that Theorem \ref{csparse} gives a positive answer to Remark 4.20 in \cite{DJS23} and it is independent of the results in that paper. So, it provides a simplification to the approach of proving Canny-Emiris' conjecture for computing $\Res_{\cA_1,\ldots, \cA_{n+1}}$ as a quotient of two determinants.

The paper is organized as follows: in Section \ref{crus} we will review basic results on toric varieties, sparse resultants, and determinants of complexes which will be of use in our proofs. We then turn to Section \ref{sec: proofthm1} to prove Theorems \ref{g1} and \ref{spr}, and finish with the proof of Theorem \ref{csparse} in Section \ref{ultt}.

\bigskip
\noindent\underline{\bf Acknowledgements:}
We thank the reviewers for their helpful and insightful comments which helped improving the presentation of this text.
Part of this work was done while the second author was visiting the University of Barcelona in July 2023, and the first author was visiting the University of Buenos Aires in December 2023. We are grateful to the support and stimulating atmosphere provided by both institutions. These research trips were funded partially by the Spanish MICINN research projects  PID2019-104047GB-I00 and PID2023-147642NB-I00. In addition, C. D'Andrea was supported by the Spanish State Research Agency, through the Severo Ochoa and Mar\'ia de Maeztu Program for Centers and Units of Excellence in R\&D (CEX2020-001084-M).
G. Jeronimo was partially supported by the Argentinian grants UBACYT 20020190100116BA and PIP 11220130100527CO CONICET.

\bigskip
\section{Background and useful tools}\label{crus}

\subsection{Toric Varieties}
 All the terminology and most of the results that we are going to use here are borrowed from \cite{CLS11}. For simplicity and abuse of notation, we will assume in this section that $K$ is an algebraically closed field of characteristic zero (what was $\overline{K}$ in the previous one).  Let $X$ be an $n$-dimensional {\em complete  toric variety}. As such, there is an $n$-dimensional lattice $N$, and $X$ is determined by a fan $\Sigma$  in $N_\R:=N\otimes\R$ such that  $|\Sigma|=N_\R.$
Let $M$ be the dual lattice of $N$, and denote with $\eta_1,\ldots, \eta_{n+r}$ ($r\geq1$) the primitive generators in $N$ of the $1$-dimensional cones in $\Sigma.$ We have then that  $r$ is the rank of the {\em Chow group} $A_{n-1}(X)$, which can be represented as the cokernel of the morphism
\begin{equation}\label{cok}
\begin{array}{ccl}
M & \to & \Z^{n+r}\\
m & \to &(\langle m,\eta_1\rangle ,\ldots, \langle m,\eta_{n+r}\rangle).
\end{array}
\end{equation}
Let $S:=K[x_1,\ldots, x_{n+r}]$ be the {\em homogeneous coordinate ring} of $X$, as defined in \cite{cox95}. As its name suggests, one can regard $x_1,\ldots, x_{n+r}$ as (homogeneous) coordinates of $X$ as follows: for each $\sigma\in\Sigma$, we set $x^{\widehat{\sigma}}:=\prod_{\eta_j\notin\sigma} x_j.$ Set  \begin{equation}\label{z}Z=V(x^{\widehat{\sigma}},\, \sigma\in\Sigma)\subset K^{n+r}.\end{equation} Each point  $\xi\in K^{n+r}\setminus Z,$  corresponds to a point $x_\xi\in X.$
It can be shown that the map $K^{n+r}\setminus Z\to X$ which sends $\xi$ into $x_\xi$ is such that makes $X$ a ``quotient'' of $K^{n+r}\setminus Z$ by the action of an $r$-dimensional group $G,$ see \cite{cox95} for more details.

As any $n$-dimensional toric variety, $X$ has  the torus $\T_{{K}}^n$ as an open dense subvariety, with coordinate ring $K[t^{\pm1}],$ and the ``trivial'' fan $\{\bf0\}.$ These coordinates are related to those given by $x_1,\ldots, x_{n+r}$ via
\begin{equation}\label{tx}
t_i=\prod_{j=1}^{n+r} x_j^{\langle\eta_j,e_i\rangle}, \ i=1,\ldots, n.
\end{equation}
From this we mean that given $\bft=(t_1,\ldots,t_n)\in\T_{{K}}^n,$ there exist $(x_1,\ldots, x_{n+r})\in K^{n+r}\setminus Z$ satisfying \eqref{tx}. So, there is a point in $X$, the quotient of this set via the group $G,$ representing $\bft.$ Note that all the $x_i$'s are actually in $K\setminus\{0\}$ so any point of $\T_{{K}}^n$ is actually lifted to a point in $\T_{{K}}^{n+r}\subset K^{n+r}\setminus Z.$

Each variable $x_j$ of $S$ corresponds to a generator $\eta_j$ and hence to a torus-invariant irreducible divisor $D_j$ of $X.$ We will  denote with $[D_j]$ the class of $D_j$ in $A_{n-1}(X),\, j=1,\ldots, n+r.$ As in \cite{cox95} we grade $S$ in such a way that the monomial
$\prod_{j=1}^{n+r}x_j^{a_j}$ has degree $\sum_{j=1}^{n+r} a_j[D_j]\in A_{n-1}(X).$
From \eqref{cok} it is clear that $\prod_{j=1}^{n+r}x_j^{a_j}$ and $\prod_{j=1}^{n+r}x_j^{b_j}$ have the same degree if and only if there is $m\in M$ such that
\begin{equation}\label{sdg}
b_j=\langle m,\eta_j\rangle+a_j, \ j=1,\ldots, n+r.
\end{equation}
Set $M_\R:= M\otimes\R$. Associated to a divisor
$
D=\sum_{j=1}^{n+r} a_j D_j,$  we can define the polytope
\begin{equation}\label{pd}
P_D:=\{u\in M_\R:\,\langle u,\eta_j\rangle\geq-a_j,\,1\leq j\leq n+r\}\subset M_\R.
\end{equation}
Note that there may be more than one way of writing $D$ as a linear combination of the $D_j$'s, so the polytope $P_D$ is well defined up to translations by elements in $M$.
By denoting with $S_\alpha$ the graded piece of $S$ of degree $\alpha=[D],$ we have that
\begin{equation}\label{kor}
S_\alpha \simeq H^0(X,\cO_X(D))=\bigoplus_{m\in P_D\cap M}K\cdot \bfchi^{m},
\end{equation}
where $\bfchi^m\in \mbox{Hom}(N\otimes_\Z K^\times, K^\times)$ is a character in $\T_{{K}}^n\simeq N\otimes_\Z K^\times,$ and $\cO_X(D)$ denotes the coherent sheaf on $X$ associated to the divisor $D$, see \cite[Proposition 1.1]{cox95}. So, any polynomial $F\in S$ of degree $\alpha$ can be written as
\begin{equation}\label{eFe}
F=\sum_{m\in P_D\cap M}c_m\prod_{j=1}^{n+r}x_j^{\langle m,\eta_j\rangle+a_j},
\end{equation}
with $c_m\in K.$ A point $x\in X$ is a {\em zero} of $F$ if there exists $\xi\in K^{n+r}$ such that $F(\xi)=0$ and $x_\xi=x.$ For a finite set $\cF$ of homogeneous elements of $S,$ we denote with $V_X(\cF)\subset X$ the set of common zeroes of all $F\in\cF.$ It is a Zariski closed set in $X.$

A divisor $D$ is {\em effective} if one can write $D=\sum_{j=1}^{n+r} b_j D_j$ with $b_j\geq0.$
Recall that (\cite[Theorem 4.2.8]{CLS11})  a divisor $D=\sum_{j=1}^{n+r} b_j D_j$ is Cartier if and only if for all $\sigma\in\Sigma$ generated by $\{\eta_{j_1},\ldots, \eta_{j_l}\},$ there exists $m_\sigma\in M$ such that $\langle m_\sigma,\eta_{j_i}\rangle=-b_{j_i}$ for $i=1,\ldots, l.$

We also have the following results which will be of use in the sequel.
\begin{proposition}\cite[Proposition 6.1.1]{CLS11}\label{usam}
A Cartier  divisor $D$  is basepoint free, i.e.~$\cO_X(D)$ is generated by its global sections, if and only if for all $n$-dimensional $\sigma\in\Sigma,\, m_\sigma\in P_D.$
\end{proposition}

A {\em very ample} divisor $D$ (\cite[Definition 6.1.9]{CLS11}) is one being basepoint free, and also such that the map
\begin{equation}\label{phialpa}
\begin{array}{cccl}
\phi_D:& X& \to & \P(H^0(X,\cO_X(D))^\vee)\\
& x &\mapsto & (\bfchi^m(x))_{m\in P_D\cap M}
\end{array}
\end{equation}
is a closed embedding. An {\em ample} divisor $D$ is one such that $\ell D$ is very ample for some integer $\ell>0.$  A {\em semi-ample} divisor is one such that $\cO_X(D)$ is generated by its global sections, or equivalently (as $X$ is a complete toric variety)  basepoint free, see \cite[Theorem 6.3.12]{CLS11}.  The following criteria for ampleness and semi-ampleness will be of use in the sequel. For a Cartier divisor $D,$ the function $\varphi_D: N\otimes\R\to\R$ such that $\varphi_D(u):=\langle m_\sigma, u\rangle$ if $u\in\sigma$, is well defined, and linear in each cone $\sigma$ (\cite[Theorem 4.2.12]{CLS11}).
\begin{proposition}[Theorems 6.1.7 and  6.1.14 in \cite{CLS11}]\label{convex} If $X$ is a complete toric variety, and $D$ a Cartier divisor on $X,$ then
\begin{enumerate}
\item $D$ is semi-ample $\iff \varphi_D$ is a  convex function.
\item $D$ is ample $\iff \varphi_D$ is a strictly convex function.
\end{enumerate}
\end{proposition}

\begin{definition}[Definition $2.2$ in \cite{BT22}]
Given a zero-dimensional subscheme $Y$ of $X,$ a degree $\alpha$ is called {\em $Y$-basepoint free} if there exists $H\in S_\alpha$  such that $Y\cap V(H)=\emptyset.$
\end{definition}

The following result is going to be useful in the text.
\begin{proposition}\label{bpf}
If $Y$ is zero-dimensional and contained in the torus $\T_{K}^n,$ then any effective $\alpha$  is $Y$-basepoint free.
\end{proposition}

\begin{proof}
Let $D$ be such that $[D]=\alpha.$ Then any polynomial in $S_\alpha$ is as in \eqref{eFe} for suitable $c_m\in K,\ m\in P_D\cap M.$
By hypothesis, any point $\bfxi\in Y$ satisfies $\prod_{j=1}^{n+r}\xi_j^{\langle m,\eta_j\rangle+a_j}\neq0.$ Each of these finitely many points amounts to a linear inequality which can always be fulfilled for a suitable choice of coefficients in $K$.
\end{proof}

A {\em $\Q$-divisor}  is an element
of the form $\sum_{j=1}^{n+r} q_j D_j,$ with $q_j\in\Q.$  It is called {\em $\Q$-ample} (resp.~{\em $\Q$-Cartier}) if a positive integer multiple of it is ample (resp.~Cartier).
We will also need the following version of the generalization of Kawamata and Viehweg vanishing theorem for toric varieties given by Musta\c{t}\u{a}. For a $\Q$-divisor of the form $D=\sum_{j=1}^{n+r}q_jD_j,$ we define
$\lceil D\rceil:=\sum_{j=1}^{n+r}\lceil q_j\rceil D_j.$

\begin{theorem}[Corollary $2.5$ in \cite{mus02}]\label{mustata}
If $D$ is a $\Q$-Cartier $\Q$-ample divisor,  then  $H^i(X,\,\cO_{X}(\lceil D\rceil+\sum_{j=1}^{n+r}-D_j))=0$ for all $i>0.$
\end{theorem}

\medskip
\subsection{The toric varieties $X_P$ and $X_{P+Q}$}\label{2.2}
Going back to our original formulation, assume that we have our input system \eqref{fi} with  associated supports $\cA_1,\ldots, \cA_k\subset\Z^n,$ and  polytopes $P_1,\ldots, P_k$ and $P$ as in \eqref{pp} being $n-$dimensional. Note that in this case we are identifying $M\simeq\Z^n.$ Let $Q$ be another integral polytope, and
denote with $\eta_1,\ldots, \eta_{n+r}$ the primitive inner normal vectors in $N\simeq\Z^n$ to the $n+r$ facets of $P+Q$. We then have  that there are integers $a_1,\ldots, a_{n+r}$ such that
\begin{equation}\label{P}
P=\{u\in\R^n:\, \langle u,\eta_j\rangle\geq-a_j,\, 1\leq j\leq n+r\}\subset\R^n,
\end{equation}
which implies that there are integers $a_{ij},\,1\leq i\leq k,\,1\leq j\leq n+r$ such that $a_j=a_{1j}+\ldots+a_{kj}$ for $j=1,\ldots, n+r,$  and
\begin{equation}\label{PPi}
P_i =\{u \in \R^n\, : \, \langle u,\eta_j \rangle \ge - a_{ij},\,\, 1\leq j\leq n+r\},\ i=1,\ldots, k.
\end{equation}
We also have that there are integers $b_1,\ldots, b_{n+r}$ such that
\begin{equation}\label{Q}
Q=\{u\in\R^n:\, \langle u,\eta_j\rangle\geq-b_j,\, 1\leq j\leq n+r\}\subset\R^n.
\end{equation}
Assume w.l.o.g. that $\eta_1,\ldots, \eta_{n+\ell}$ are the primitive inner normals to the facets of $P,$ and let $X_P$ be the toric variety associated with $P,$ i.e. the one  whose fan is the \textit{normal fan} $\Sigma_P$  defined in Chapter $2$ of \cite{CLS11}, where the maximal cones of this fan are in correspondence with the vertices of $P,$ and the cone associated to a vertex $v$ is generated by the rays generated by those $\eta_j$ such that the facet determined by $\eta_j$ contains $v$. In the same way the variety $X_{P+Q}$ is defined, by using the fan $\Sigma_{P+Q}$. Note that this fan refines $\Sigma_P$ and hence, thanks to Theorem $3.4.11$ in \cite{CLS11},  there is a toric proper morphism $X_{P+Q}\to X_P.$

For $i=1,\ldots, k$ we set $D_{P_i}:=\sum_{j=1}^{n+r}a_{ij}D_j$ and $\alpha_{P_i}:=\sum_{j=1}^{n+r} a_{ij}[D_j]$ the divisor associated to $P_i$ in $X_{P+Q}$ and its respective degree in the Chow group of this variety. Note that
$D_P=\sum_{i=1}^k D_{P_i}$ and $\alpha_P=\sum_{i=1}^k \alpha_{P_i}.$ Set also $D_Q=\sum_{j=1}^{n+r} b_j D_j$ and $\alpha_Q=[D_Q].$

As explained in \cite[\S $2.2$]{KS05}, for any integral polytope $P_i,$ the associated piece-wise linear function $\varphi_{D_{P_i}}:\R^n\simeq N\otimes\R\to\R$ defined by using its normal fan $\Sigma_{P_i}$ is convex, this is because $D_{P_i}$ is basepoint free in $X_{P_i}$ thanks to Proposition \ref{usam}.
The same applies to $D_P$ in $X_P$ and $D_Q$ in $X_Q$. As the fan $\Sigma_{P+Q}$ is a common refinement of $\Sigma_{P_i},\,\Sigma_P,\,\Sigma_Q$, it turns out that the pull-back of all these divisors to $X_{P+Q}$  has the same  (convex) piece-wise linear function, and define the same polytope.  This implies then that all the divisors $D_{P_i},\,1\leq i\leq k,\, D_P$ and $D_Q$ are Cartier and semi-ample in $X_{P+Q}.$ From \cite[Proposition 6.1.10]{CLS11}, we have that $D_{P+Q}:=D_P+D_Q$ is in addition ample in $X_{P+Q}.$ Note that neither $D_P$ nor $D_Q$ nor any of the $D_{P_i}$ is necessarily  ample.

As $\cA_i\subset\Z^n\cap P_i,$ each of the $f_i$'s  from \eqref{fi} can be  both ``$P-$homogeneized'' and ``$(P+Q)$-homogeneized''  to $F_i$  and $F_i'$ respectively, using \eqref{kor} and \eqref{eFe} as follows:
\begin{equation}\label{eFFe}
\begin{array}{ccl}
F_i(x_1,\ldots, x_{n+\ell})&=&\sum_{b\in\cA_i} f_{i,b}\prod_{j=1}^{n+\ell}x_j^{\langle b,\eta_j\rangle+a_{ij}}=\Big(\prod_{j=1}^{n+\ell}x_j^{a_{ij}}\Big)f_i(t_1,\ldots, t_n),\\
F_i'(x_1,\ldots, x_{n+r})&=&\sum_{b\in\cA_i} f_{i,b}\prod_{j=1}^{n+r}x_j^{\langle b,\eta_j\rangle+a_{ij}}=\Big(\prod_{j=1}^{n+r}x_j^{a_{ij}}\Big)f_i(t_1,\ldots, t_n),
\end{array}
\end{equation}
the right hand side equalities  due to \eqref{tx}. From here, it is clear that if $\bft\in\T_{\overline{K}}^n$ is such that $f_i(\bft)=0,$ then $F_i(x_\bft)=F'_i(x'_\bft)=0$ for the corresponding points of $\bft$ in the varieties $X_P$ and $X_{P+Q}$ respectively.   It turns out that the closure of the zeroes of $f_i$ in $X_P$ (resp. $X_{P+Q}$) is the variety of zeroes of $F_i$ (resp. $F_i'$).

\begin{example}
{\em
For our running Example \ref{running}, the heptagon $P$ has the following inner normals, numbered clockwise, se Figure \ref{figoo}:
$$\begin{array}{l}
\eta_1=(1,0),\,\eta_2=(2,-1),\, \eta_3=(1,-1),\, \eta_4=(0,-1),\,\eta_5=(-1,-1),\\ \eta_6=(-1,1),\,\eta_7=(-1,2).
\end{array}
$$}

\begin{figure}\label{polifan}
\begin{tikzpicture}
\filldraw[color=gray!40] (0,0)--(1,0)--(-0.45,0.9)--(0,0);
\filldraw[shift={(0,1)},color=gray] (0,0)--(1,0)--(0.9,-0.45)--(0,0);
\filldraw[shift={(1,3)},color=gray!70] (0,0)--(0.7,-0.7)--(0.9,-0.45)--(0,0);
\filldraw[shift={(2,4)},color=gray!40] (0,0)--(0.7,-0.7)--(0,-1)--(0,0);
\filldraw[shift={(4,4)},color=gray] (0,0)--(-0.7,-0.7)--(0,-1)--(0,0);
\filldraw[shift={(5,3)},color=gray!40] (0,0)--(-0.7,-0.7)--(-0.7,0.7)--(0,0);
\filldraw[shift={(4,2)},color=gray!70] (0,0)--(-0.7,0.7)--(-0.45, 0.9)--(0,0);
\filldraw (0,0) circle (2pt);
\filldraw (0,1) circle (2pt);
\filldraw (1,3) circle (2pt);
\filldraw (2,4) circle (2pt);
\filldraw (4,4) circle (2pt);
\filldraw (5,3) circle (2pt);
\filldraw (4,2) circle (2pt);
\draw (2.5,0.5) node {$P$};
\draw[very thick] (0,0)--(0,1)--(1,3)--(2,4)--(4,4)--(5,3)--(4,2)--(0,0);
\draw[shift={(0,0)},->, very thick] (0,0)--(1,0);
\draw[shift={(0,1)},->, very thick] (0,0)--(1,0);
\draw[shift={(0,1)},->, very thick] (0,0)--(0.9,-0.45);
\draw[shift={(1,3)},->, very thick] (0,0)--(0.9,-0.45);
\draw[shift={(1,3)},->, very thick] (0,0)--(0.7,-0.7);
\draw[shift={(2,4)},->, very thick] (0,0)--(0.7,-0.7);
\draw[shift={(2,4)},->, very thick] (0,0)--(0,-1);
\draw[shift={(4,4)},->, very thick] (0,0)--(0,-1);
\draw[shift={(4,4)},->, very thick] (0,0)--(-0.7,-0.7);
\draw[shift={(5,3)},->, very thick] (0,0)--(-0.7,-0.7);
\draw[shift={(5,3)},->, very thick] (0,0)--(-0.7,0.7);
\draw[shift={(4,2)},->, very thick] (0,0)--(-0.7,0.7);
\draw[shift={(4,2)},->, very thick] (0,0)--(-0.45,0.9);
\draw[shift={(0,0)},->, very thick] (0,0)--(-0.45,0.9);
\end{tikzpicture}
\hspace{1cm}
\begin{tikzpicture}[scale=0.8]
\filldraw[color=gray] (0,0)--(3,0)--(3,-1.5)--(0,0);
\filldraw[color=gray!70] (0,0)--(3,-1.5)--(3,-3)--(0,0);
\filldraw[color=gray!40] (0,0)--(3,-3)--(0,-3)--(0,0);
\filldraw[color=gray] (0,0)--(0,-3)--(-3,-3)--(0,0);
\filldraw[color=gray!40] (0,0)--(-3,-3)--(-3,3)--(0,0);
\filldraw[color=gray!70] (0,0)--(-3,3)--(-1.5,3)--(0,0);
\filldraw[color=gray!40] (0,0)--(3,0)--(3,3)--(-1.5,3)--(0,0);
\draw[->, color= darkgray] (-3,0) -- (3,0);
\draw[->, color= darkgray] (0,-3) -- (0,3);
\draw[->,very thick] (0,0)--(1,0) node[above] {$\eta_1$};
\draw[thick] (0,0)--(3,0);
\draw[->, very thick] (0,0)--(2,-1) node[above] {$\eta_2$};
\draw[thick] (0,0)--(3,-1.5);
\draw[->, very thick] (0,0)--(1,-1) node[below] {$\eta_3$};
\draw[thick] (0,0)--(3,-3);
\draw[->, very thick] (0,0)--(0,-1) node[left] {$\eta_4$};
\draw[thick] (0,0)--(0,-3);
\draw[->, very thick] (0,0)--(-1,-1) node[above] {$\eta_5$};
\draw[thick] (0,0)--(-3,-3);
\draw[->, very thick] (0,0)--(-1,1) node[left] {$\eta_6$};
\draw[thick] (0,0)--(-3,3);
\draw[->, very thick] (0,0)--(-1,2) node[right] {$\eta_7$};
\draw[thick] (0,0)--(-1.5,3);
\end{tikzpicture}
   \caption{Inner normals and normal fan of the polygon $P$ of  Example  \ref{running}.}  \label{figoo}
\end{figure}
{\em
The normal fan associated to this polygon has the seven maximal cones covering $\R^2$ generated by $\eta_1,\,\eta_2;$\, $\eta_2,\,\eta_3;\ldots;\, \eta_6,\eta_7;$ and $\eta_7,\,\eta_1.$ The toric variety $X_P$ is then a toric surface made by gluing seven affine spaces all of them equivalent to $\A_K^2$ except the one associated to the cone $\eta_7,\,\eta_1$ which is a quotient of this space via the action of a finite group. So it is not smooth. For more on toric surfaces, we refer the reader to Chapter $10$ of \cite{CLS11}.

Its homogeneous coordinate ring is then $S=K[x_1, x_2, x_3, x_4, x_5, x_6, x_7],$ and  one can easily verify that the variety $Z$ of \eqref{z} is equal to
$$V(x_3x_4x_5x_6x_7,\,x_1x_4x_5x_6x_7,\,x_1x_2x_5x_6x_7,\,x_1x_2x_3x_6x_7,\,x_1x_2x_3x_4x_7,\,x_1x_2x_3x_4x_5,\,x_2x_3x_4x_5x_6).
$$

By setting $Q=\{\bf0\},$ the homogenization \eqref{eFFe}  for this example is:}
$$\begin{array}{ccl}
F_1&=&  f_{10}\,x_4x_5^2+f_{11} \,x_1x_2x_7, \\
F_2&=&  f_{20}\,x_3x_4^2x_5^3x_6+f_{21}\,x_1^2x_2^3x_3^2x_4+f_{22}\,x_1x_6^2x_7^3,  \\
F_3&=& f_{30}\,x_2x_3x_4x_5^3x_6+f_{31}\,x_1^2x_2^4x_3^2+f_{32}\,x_5^2x_7^2. 
\end{array}
$$

\end{example}

\bigskip
As in the Introduction, pick $\delta\in N\otimes\Q\simeq\Q^n,$ and set
$D_\delta:=\sum_{\langle\delta,\eta_j\rangle>0}-D_j.$ Note that for any $r>0,$ we have $D_{\delta}=D_{r\delta}.$
\begin{proposition}\label{must}
For $\delta\in\Q^n,$  we have
$$H^i\big(X_{P+Q},\,\cO_{X_{P+Q}}(D_\delta+D_Q+\sum_{j\in J}D_{P_j})\big)=0\ \forall J\subset\{1,\ldots, k\} \ \forall i>0.$$
\end{proposition}

\begin{proof}
Assume w.l.o.g.~that $D_\delta=\sum_{\langle\delta,\eta_j\rangle>0}-D_j=\sum_{j=1}^d -D_j$ for $1\leq d< n+r.$
Fix $J\subset\{1,\ldots, k\}.$ Recall that $D_P+D_Q$ is an ample divisor in $X_{P+Q},$
and write
\begin{equation}\label{cj}
D_P+D_Q=\sum_{j=1}^{n+r} c_jD_j.
\end{equation}
As $P+Q$ is an $n$-dimensional integral polytope, there must be a monomial in $S_{\alpha_P+\alpha_Q}$ of the form $x_1^{c_1}\ldots x_{n+r}^{c_{n+r}},$ with $c_j\geq0$ and $c_j>0$ if $\langle \delta,\eta_j\rangle=0$ (for instance, one can  pick a monomial corresponding via \eqref{kor} to any vertex of $P+Q$ not contained in a facet with inner normal $\eta_j$ such that  $\langle \delta, \eta_j \rangle = 0$). We assume w.l.o.g.~that the data $(c_j)_{1\leq j\leq n+r}$ satisfies these conditions.

For any $z\in\Z_{<0}$ such that $z\delta\in\Z^n,$ thanks to \eqref{sdg} we have that  
$D_P+D_Q=\sum_{j=1}^{n+r} (c_j+\langle z\delta, \eta_j\rangle)D_j.$ If we pick $|z|$ large enough, we can assume then w.l.o.g.~that $c_j<0$ for $j\leq d$ and $c_j>0$ otherwise. Let $c>0$ be larger than the maximum of the absolute values of all the $c_j$'s, and set
$$D:=\frac{1}{c}(D_P+D_Q)+D_Q+\sum_{j\in J}D_{P_j}.$$
Note that $D$ is $\Q$-Cartier, and we clearly have that
$$\lceil D\rceil +\sum_{j=1}^{n+r}-D_j= D_\delta+D_Q+\sum_{j\in J}D_{P_j}.
$$
The claim will follow thanks to Theorem \ref{mustata} provided that  $D$  is $\Q$-ample. To show this, note that  both $D^1:=c\left(\frac{1}{c}(D_P+D_Q)\right)$ and $D^2:=c(D_Q+\sum_{j\in J}D_{P_j})$ are the divisors associated to integral polytopes, the first of them being equal to $P+Q.$  From the properties of these divisors explained in Section \ref{2.2}, we conclude that $D^1$ is ample, and $D^2$ semi-ample in $X_{P+Q}.$ Hence, thanks to  Proposition \ref{convex},   $\varphi_{D^1}$ is strictly convex, and $\varphi_{D^2}$ convex. We deduce then that $\varphi_{D^1+D^2}$ is the sum of a strictly convex function plus a  convex function. So, it is a strictly convex function, which  implies that $D^1+D^2=c\,D$ is ample thanks again to Proposition \ref{convex}.
This concludes with the proof of the claim.
\end{proof}

\bigskip
\subsection{Sparse Resultants}\label{sr}
We review here some results on sparse resultants. We will mainly follow \cite{DS15} and \cite{DJS23}.
Given finite sets $\cA_1,\ldots, \cA_{n+1}\subset\Z^n$ of respective cardinality $r_1,\ldots, r_{n+1}$  consider the generic polynomials $f^g_1,\ldots, f^g_{n+1}$ defined in \eqref{gfi}. They are elements of $\K[t^{\pm}]$ but also of $\Z[c_{i,b}][t^{\pm1}].$

Given an algebraically closed field $K$ of characteristic zero, the {\em incidence variety} of the family \eqref{gfi} is
$$\Omega_{\cA_1,\ldots, \cA_{n+1}}=\{(\tilde{c}_{i,b},\xi):\,f^g_i(\tilde{c}_{i,b},\xi)=0,\,1\leq i\leq n+1\}\subset\P_K^{r_1-1}\times\ldots\times\P_K^{r_{n+1}-1}\times\T_{{K}}^n,
$$
which is an irreducible subvariety of codimension $n+1$ defined over $\Q.$ Denote with
\begin{equation}\label{pi}\pi:\P_K^{r_1-1}\times\ldots\times\P_K^{r_{n+1}-1}\times\T_{{K}}^n\to\P_K^{r_1-1}\times\ldots\times\P_K^{r_{n+1}-1},
\end{equation} the projection which forgets the last factor. If the Zariski closure of the image $\overline{\pi(\Omega_{\cA_1,\ldots,\cA_{n+1}})}$  is a hypersurface, the  {\it sparse eliminant} $\Elim_{\cA_1,\ldots, \cA_{n+1}}$ is defined as any irreducible polynomial in $\Z[c_{i,b}]$ giving an equation for it. Otherwise, the eliminant is set equal to $1.$  The sparse resultant, denoted by $\Res_{\cA_1,\ldots, \cA_{n+1}},$ is defined as any primitive polynomial in $\Z[c_{i,b}]$ giving an equation for the direct image $\pi_*(\Omega_{\cA_1,\ldots, \cA_{n+1}}).$

Both of these polynomials are well defined up to a sign. Moreover $\Elim_{\cA_1,\ldots, \cA_{n+1}}$ is irreducible when it is not equal to $1,$ and there exists $d\in\N$ such that $\Res_{\cA_1,\ldots, \cA_{n+1}}=\Elim_{\cA_1,\ldots, \cA_{n+1}}^d.$

The following properties  of eliminants and resultants will be of use in the text.
\begin{proposition}\label{ress} $^{}$
\begin{enumerate}
\item For $i=1,\ldots, n+1,\,\Res_{\cA_1,\ldots, \cA_{n+1}}$ is a homogeneous polynomial in the coefficients of $f^g_i$ of degree $MV(P_1,\ldots, \check{P_i},\ldots, P_{n+1}),$ where $P_j=\conv(\cA_j)$ for $j=1,\dots, n$, and $MV(\cdot )$ denotes the \emph{mixed volume} of $n$ polytopes as defined in \cite[(2.29)]{DS15}.
\item Let $f_1,\ldots, f_{n+1}\in K[t^{\pm1}]$ be such that their supports are contained in $\cA_1,\ldots, \cA_{n+1}$ respectively, and $F_1,\ldots, F_{n+1}\in S$ (resp. $F_1',\ldots, F_{n+1}'\in S$) their homogeneizations with respect to $P_1,\ldots, P_{n+1}$ via \eqref{eFFe}. If $\Elim_{\cA_1,\ldots, \cA_{n+1}}$ is not equal to $1,$ then, $\Elim_{\cA_1,\ldots, \cA_{n+1}}(f_1,\ldots, f_{n+1})=0\iff V_{X_P}(F_1,\ldots, F_{n+1})\neq\emptyset\iff V_{X_{P+Q}}(F_1',\ldots, F_{n+1}')\neq\emptyset.$
\end{enumerate}
\end{proposition}

\begin{proof} $^{}$
\begin{enumerate}
\item Follows from Proposition $3.4$ in \cite{DS15}.
\item If  $\overline{\pi(\Omega_{\cA_1,\ldots,\cA_{n+1}})}$ is a hypersurface, then we have from Proposition 3.2 in \cite{DS15}  and (2.14) in loc. cit. that $\Elim_{\cA_1,\ldots, \cA_{n+1}}(f_1,\ldots, f_{n+1})=0$ if and only if the ``homogenization'' of this system to  the toric multiprojective variety $X_{\cA_1,\ldots, \cA_{n+1}}$ defined in \cite[(2.17)]{DS15}, has a common zero in this set. As there is a proper morphism of toric varieties $\phi:X_P\to\P^{r_1-1}\times\ldots\times\P^{r_{n+1}-1}$ satisfying $\phi(X_P)=X_{\cA_1,\ldots,\cA_{n+1}}$  (cf. Lemma $2.6$ in \cite{DS15}), the claim follows straightforwardly for $X_P$ by noting that the $F_i$'s are global sections of divisors in $X_P$ corresponding to the homogenizations of the $f_i$'s in $X_{\cA_1,\ldots, \cA_{n+1}}$. The same argument applies to $X_{P+Q}.$
\end{enumerate}
\end{proof}

\begin{remark}
There is a combinatorial criteron on $\cA_1,\ldots, \cA_{n+1}$ to detect when $ \Elim_{\cA_1,\ldots, \cA_{n+1}}=1,$ see Proposition $3.8$ in \cite{DS15} for instance.
\end{remark}

\medskip
\subsection{Determinants of complexes}\label{detc}
We present here the notion of determinants of exact and generically exact complexes, and some of their properties which will be of use to prove Theorem \ref{csparse}.
We will follow mainly the presentation given in \cite[Appendix A]{GKZ94}, see also \cite{dem84,cha93} for more on this subject.

Let $K$ be any field. A {\em complex} of  $K$-vector spaces is a graded vector space $\bigoplus_{i\in\Z} V_i$ together with $K$-linear maps $
d_i:V_i\to V_{i-1}$ satisfying $d_i\circ d_{i+1}=0.$
We will assume that all but finitely many of the $V_i$'s are different from zero.

For an $n$-dimensional $K$-vector space $V$ of dimension $n$, we define $\mbox{Det}(V)$ as the one-dimensional vector space $\wedge^nV$, and for $V=0$, we set $\mbox{Det}(0)=K.$  Given a one-dimensional vector space $V$, we denote with $V^{-1}$ its dual space $V^*.$
The {\em determinantal vector space} of a finite-dimensional complex $\bigoplus_{i\in\Z} V_i$  of vectors spaces having only finitely many of the $V_i$'s  different from zero is the one-dimensional $K$-vector space defined as  $\bigotimes _i \mbox{Det}(V_i)^{\otimes (-1)^i}.$ The cohomology spaces of a complex form another graded vector space and there is a natural isomorphism between the determinantal vector spaces associated to both (see \cite[Proposition $3$, Appendix A]{GKZ94}).
If the complex is exact,  its determinantal vector space can be naturally identified with $K$ (\cite[Corollary $6$, Appendix A]{GKZ94}). To make explicit this identification, one chooses bases $B_i$ of all the nontrivial $V_i$'s and, as in the linear case, defines the {\em determinant of $\bigoplus_{i\in\Z} V_i$ with respect to the system of bases $\{B_i\}$}
as the element of $K^\times$ which is the image under this natural identification  of the element $\bigotimes_i \mbox{Det}(B_i)^{(-1)^i} \in \mbox{Det}(\bigoplus_{i\in\Z} V_i)$, where $\mbox{Det}(B_i)$ is the wedge product of all the elements in $B_i,$
 (cf. \cite[Definition $7$, Appendix A]{GKZ94}). In the case corresponding to only two nontrivial vector spaces $0\to V_2\stackrel{d_2}{\to} V_1\to0,$ this notion coincides with the determinant of the matrix of $d_2$ in the bases $B_2$ and $B_1.$

In general, to compute this determinant once the bases are fixed, we can use the so-called ``Cayley method'': suppose that
\begin{equation}\label{fcs}
0\to V_r\stackrel{d_r}{\to} V_{r-1}\stackrel{d_{r-1}}{\to}\ldots\stackrel{d_2}{\to} V_1\stackrel{d_1}{\to} V_0\to0
\end{equation}
is an exact complex of finite-dimensional vector spaces $V_0,\dots, V_r$ with bases $B_0,\ldots, B_r$ respectively and,  for $i=1,\dots, r$, let $M_i$ be the matrix of $d_i$ with respect to the bases $B_i$ and $B_{i-1}$.  A family of subsets $I_i\subset B_i,$ $i=0,\dots, r$ is called {\em admissible} (\cite[Definition $12$, Appendix A]{GKZ94}) if $I_0=B_0,\, I_r=\emptyset,$ and
for any $i=1,\ldots, r,$ we have $\#(B_{i}\setminus I_{i})=\#(I_{i-1}),$ and the submatrix of $M_i$ having its columns indexed by $B_{i}\setminus I_{i}$, and its rows indexed by $I_{i-1}$ is invertible. We have that admissible sets do exist (\cite[Proposition $13$, Appendix A]{GKZ94}), and moreover:
\begin{proposition}[Theorem 14, Appendix A in \cite{GKZ94}]\label{alt}
For any admissible collection $(I_i)_{i=0,\ldots, r}$, for $i=1,\dots, r$, denote with $\Delta_i$ the determinant of the submatrix of $M_{i}$ having its columns indexed by $B_{i}\setminus I_{i}$, and its rows indexed by $I_{i-1}.$ Then, the determinant of \eqref{fcs} with respect to the bases $B_0,\ldots, B_r$ is equal to
$$\prod_{i=1}^{r} \Delta_i^{(-1)^{i+1}}.
$$
\end{proposition}

We will be interested in computing the determinant of complexes of $K$-vector spaces once these bases are fixed by using the  ``descending method'' (\cite{dem84}, see also \cite{DJ05}), which consists of the following algorithm:
\begin{enumerate}
\item Set $I_0=B_0.$
\item For $i=1,\ldots, r,$ choose $I_{i}$ as any subset of $B_{i}$ such that the submatrix of $M_{i}$ indexed by the columns in $B_{i}\setminus I_{i}$ and the rows in $I_{i-1}$ is of maximal rank.
\end{enumerate}
As the complex is exact, we deduce straightforwardly that if this algorithm finishes successfully, then $I_r=\emptyset.$

The fact that the algorithm succeeds in producing an admissible sequence follows from the following claim.

\begin{proposition}
For $i=0,\dots, r$, the set $I_i$ obtained by the above algorithm satisfies $\ker(d_{i}) \oplus \langle B_{i}\setminus I_i\rangle =V_i$.
\end{proposition}

\begin{proof}
We proceed by induction. The identity holds trivially for $i=0$.

Assume $i>0$ and $\ker(d_{i-1}) \oplus \langle B_{i-1}\setminus I_{i-1}\rangle =V_{i-1}$. Let $\pi_{i-1}: V_{i-1} \to \langle I_{i-1}\rangle$ be the projection with kernel $\langle B_{i-1}\setminus I_{i-1}\rangle$
and  $\tilde d_i=\pi_{i-1} \circ d_i: V_i \to \langle I_{i-1}\rangle$. Note that the matrix of $\tilde d_i$ with respect to the bases $B_i$ and $I_{i-1}$ is the submatrix of $M_i$ with rows indexed by $I_{i-1}$. We have $\ker(\tilde d_i) =\ker(d_i)$, since $\mbox{Im}(d_i) \cap \langle B_{i-1}\setminus I_{i-1}\rangle = \ker (d_{i-1}) \cap \langle B_{i-1}\setminus I_{i-1}\rangle =\{0\}$ because of the induction assumption and the fact that the complex is exact. Then, $\dim(\mbox{Im}(\tilde d_i)) =\dim(\mbox{Im}(d_i))= \dim(\ker (d_{i-1}))= \# I_{i-1}$ and $\tilde d_i$ is onto.
We conclude that the algorithm chooses a subset $I_i$ of $B_i$ such that $\#(B_i \setminus I_i) = \# I_{i-1}$ and the submatrix of $\tilde d_i$ with columns indexed by $B_i \setminus I_i$ is invertible. In particular, $ \ker(\tilde d_i) \cap \langle B_i \setminus I_i\rangle = \{0\}$ and, comparing dimensions, we deduce that $\ker(d_{i}) \oplus \langle B_{i}\setminus I_i\rangle = V_i$.
\end{proof}

Let now $R=A[y_1,\ldots, y_s]$ be a ring of polynomials in the indeterminates $y_1,\ldots, y_s,$ with coefficients in a noetherian integral domain $A$, such that $R$ is factorial and regular (i.e. every projective $R$-module has a projective finite resolution). Denote with $K(R)$ its field of fractions.  A finite {\em generically exact complex} of $R$-modules of finite rank is a complex of free $R$-modules
\begin{equation}\label{cfx}
0\to \mathbf{M}_r\stackrel{d_r}{\to} \mathbf{M}_{r-1}\stackrel{d_{r-1}}{\to}\ldots\stackrel{d_1}{\to} \mathbf{M}_0\to0
\end{equation}
such that their tensorization with $K(R)$ makes the resulting complex of finite dimensional $K(R)$-vector spaces exact. Given a system of $R$-bases $B_0,\dots, B_r$ of the free $R$-modules involved, the determinant of this complex is then defined as the determinant of  \eqref{cfx} tensored with $K(R).$  It is an element of $K(R)^\times.$
In some situations, one can compute it as the $\gcd$ of the maximal minors of the matrix of the map $d_1$ in the corresponding bases. To state the conditions ensuring that such property holds,
as usual we define the {\em cohomology modules} of  \eqref{cfx} as $H^i(\bigoplus_j \mathbf{M}_j):=\ker(d_{i})/ \mbox{Im}(d_{i+1}), i\in\Z$.
We will also need the notion of \emph{multiplicity} of a finitely generated $R$-module $\mathbf{M}$ along a prime ideal $\mathfrak{p}$ of $R$, which is defined, via localization at $\mathfrak{p}$, as $\mbox{mult}_{\mathfrak{p}}(\mathbf{M}) = \sum_{\ell\ge 0}
\dim_{k_\mathfrak{p}} \big(\mathfrak{m}_{\mathfrak{p}}^\ell \mathbf{M}_\mathfrak{p}/\mathfrak{m}_{\mathfrak{p}}^{\ell+1} \mathbf{M}_\mathfrak{p}\big)$ if the sum is finite or $0$ otherwise. Here, $\mathfrak{m}_\mathfrak{p}$ denotes the maximal ideal of the local ring $R_\mathfrak{p}$ and $k_\mathfrak{p}= R_{\mathfrak{p}}/\mathfrak{m}_\mathfrak{p}$ its residual field.
We have the following result.

\begin{theorem}\cite[Theorem $34$, Appendix A]{GKZ94}\label{mult}
If for any prime element $p\in R,$  the multiplicity of $H^i(\bigoplus_j \mathbf{M}_j)$ along the ideal $\langle p\rangle$ of $R$ is equal to zero $\forall i>0$, then --up to an invertible element in $R$--  the determinant of \eqref{cfx} with respect to the bases $B_0,\ldots, B_r$ is equal to the $\gcd$ of the maximal minors of the matrix of $d_1$ in the bases $B_1$ and $B_0$.
\end{theorem}

\bigskip
\section{Proof of Theorems \ref{g1} and \ref{spr}}\label{sec: proofthm1}

To prove Theorem \ref{g1}, assume first that $K$ is algebraically closed. \par
\noindent{$(1)\implies (2)$}
Suppose first that $V_{X_{P}}(F_1,\ldots, F_k)=\emptyset,$ which is equivalent to  having  $V_{X_{P+Q}}(F_1',\ldots, F_k')=\emptyset.$
As $F'_i$ is a global section of $\cO_{X_{P+Q}}(D_{P_i}),\, 1\leq i\leq k,$ we get the Koszul complex of sheaves
 \begin{equation}\label{cx1}
\begin{array}{r}
0\to\cO_{X_{P+Q}}\big(-\sum_{1\leq i\leq k} D_{P_i}\big)\to \ldots \to\bigoplus_{|J|=\ell} \cO_{X_{P+Q}}(-\sum_{j\in J}D_{P_j})\to\ldots\\[2mm]
\dots \to\bigoplus_{1\leq i\leq k} \cO_{X_{P+Q}}(-D_{P_i})\to\cO_{X_{P+Q}}\to0
\end{array}
\end{equation}
which is exact since each $\cO_{X_{P+Q}}(D_{P_i})$ is locally free and $F_1',\ldots, F_k'$ define the empty variety in $X_{P+Q}.$

If we now tensor \eqref{cx1} with  $\cO_{X_{P+Q}}(D_\delta+D_Q+D_P),$ and use that  each of the divisors appearing in \eqref{cx1} is Cartier, we get a new complex
\begin{equation}\label{cx2}
\begin{array}{r}
0\to\cO_{X_{P+Q}}\big(D_\delta+D_Q\big)\to   \ldots \to\bigoplus_{|J|=\ell} \cO_{X_{P+Q}}(D_\delta+D_Q+\sum_{j\notin J}D_{P_j})\to\ldots
\\[2mm]
\ldots \to\bigoplus_{1\le i \le k} \cO_{X_{P+Q}}(D_\delta+D_Q+D_P-D_{P_i})\to\cO_{X_{P+Q}}(D_\delta+D_Q+D_P)\to0
\end{array}
\end{equation}
which is also exact because all sheaf Tor groups vanish when one of the factors is locally free, which is the case with the factors of \eqref{cx1}, see the proof of Theorem 2.2 in \cite{CD05} for more. Thanks to Proposition \ref{must}, we have that, for $j>0,$
 and $J\subset\{1,\ldots, k\},$ the following cohomology vanish:
$$H^j \Big(X_{P+Q},\bigoplus_{J}\cO_{X_{P+Q}}\big(D_\delta+D_Q+\sum_{i\notin J} D_{P_i}\big)\Big)\simeq
\bigoplus_{J} H^j \Big(X_{P+Q},\cO_{X_{P+Q}}\big(D_\delta+D_Q+\sum_{i\notin J} D_{P_i}\big)
\Big).
$$
This implies straightforwardly that taking the global sections of \eqref{cx2} preserves exactness. We use \eqref{kor} to identify
$H^0 \left(X_{P+Q},\cO_{X_{P+Q}}\big(D_\delta+D_Q+\sum_{i\notin J} D_{P_i}\big)\right)$ with the elements in $S'$ of degree $\alpha_\delta+\alpha_Q+\sum_{i\notin J}\alpha_{P_i},$ so we have that the following complex of $K$-vector spaces is exact:
\begin{equation}\label{cx3}
0\to S'_{\alpha_\delta+\alpha_Q}\to \ldots \to\bigoplus_i  S'_{\alpha_\delta+\alpha_Q+\alpha_P-\alpha_{P_i}}\to S'_{\alpha_\delta+\alpha_Q+\alpha_P}\to0.
\end{equation}
This is the $(\alpha_\delta+\alpha_Q+\alpha_P)-$graded piece of the Koszul complex  $Kos(S',F_1',\ldots, F'_k),$ which proves the ``if'' part.

\noindent{$(2)\implies (3)$} Trivial.

\noindent{$(3)\implies (1)$} Let $f_1^g,\ldots, f_k^g$ be the generic polynomials defined in \eqref{gfi}, and denote with $F_1^g{'},\ldots, F_ k^g{'}$ their respective homogenizations in $S'_\K$ with respect to $P_1,\ldots, P_k.$ For a particular specialization of coefficients
$c_{i,a}\mapsto y_{i,a}$ in $K,$ we abbreviate with $F_{i,y}\in K[x_1,\ldots, x_{n+r}]=:S^0$ the specialized polynomial. Let $\cV$ be the set in the space of coefficients consisting of all $y_{i,a}$ such that the system $F_{1,y},\ldots, F_{k,y}$ has a common zero in $\T_{K}^{n+r}$. Then, no monomial in $S^0_{\alpha_\delta+\alpha_Q+\alpha_P}$  can be in the image of $\Psi_{1,\alpha_\delta+\alpha_Q+\alpha_P},$ the rightmost nontrivial map in
 $Kos(S^0,F_{1,y},\ldots, F_{k,y})_{\alpha_\delta+\alpha_Q+\alpha_P},$   as otherwise by specializing the variables $x_i$ in the projective coordinates of the common zero we would get a contradiction. Thus, for any system contained in $\cV,\, \Psi_{1,\alpha_\delta+\alpha_Q+\alpha_P}$ is not onto. We conclude that $\cV$ is contained in the algebraic variety defined by the vanishing of all maximal minors of the rightmost nontrivial map in the complex
 $Kos(S'_\K,F_1^g{'}, \ldots, F_k^g{'})_{\alpha_\delta+\alpha_Q+\alpha_P},$ and hence the Zariski closure of $\cV$ is also contained in this variety. Since  $\T_{K}^{n+r}$ is dense in $K^{n+r}\setminus Z,$ where $Z$ has been defined in \eqref{z} by using the fan $\Sigma_{P+Q},$  and the toric variety $X_{P+Q}$ is the quotient of $K^{n+r}\setminus Z$ via the action of a group (cf. \cite[\S 5.1]{CLS11}), we have that the Zariski closure of $\cV$ contains those systems $F_1',\ldots, F_k'$ of degrees $\alpha_{P_1},\ldots, \alpha_{P_k}$ respectively having a common zero in $X_{P+Q},$ or equivalently systems $F_1,\ldots, F_k$ having a common zero in $X_P.$  For all these families, the rightmost nontrivial map in $Kos(S',F_1',\ldots, F_k')_{\alpha_\delta+\alpha_Q+\alpha_P}$ cannot be onto for all the explained above. This concludes with the proof of Theorem \ref{g1} for an algebraically closed field $K$.

 For the general case, as each of the vector spaces in $Kos(S_{\overline{K}}',F_1',\ldots, F_k')_{\alpha_\delta+\alpha_Q+\alpha_P}$  has a monomial basis, which is also a monomial basis of the corresponding space in $Kos(S_K',F_1',\ldots, F_k')_{\alpha_\delta+\alpha_Q+\alpha_P},$ and the matrices of the linear maps in both complexes (w.r.t. these bases) have their entries in $K$ (in fact they are the same matrices), the claim follows straightforwardly for $Kos(S_K',F_1',\ldots, F_k')_{\alpha_\delta+\alpha_Q+\alpha_P}$ as it holds for the first complex.
   \qed

\smallskip
\begin{proof}[Proof of Theorem \ref{spr}]
The hypothesis on the coordinates of $\delta$ implies that there is a one to one correspondence via \eqref{kor} between the monomials in  $S'_{\alpha_\delta+\alpha_Q+\alpha_P}$  and the elements of
 $(P+Q+\delta)\cap\Z^n$. The same applies to the monomials in $S'_{\alpha_\delta+\alpha_Q+\alpha_P-\alpha_{P_i}}$  and the elements of  $(P_1+\ldots +P_{i-1}+P_{i+1}+\ldots+P_k+Q+\delta)\cap\Z^n, \ 1\leq i\leq k.$

The claim is now a restatement of $(1)\iff (3)$ in Theorem \ref{g1} after dehomogenizing the rightmost nontrivial map of $Kos(S',F_1',\ldots, F_k')_{\alpha_\delta+\alpha_P+\alpha_Q}$ to the torus via \eqref{tx}.
\end{proof}
\bigskip
\section{Proof of Theorem \ref{csparse}}\label{ultt}
As $\pi(\Omega_{\cA_1,\ldots, \cA_{n+1}})$ has positive codimension in its codomain (see \eqref{pi}),  then
$F^g_1,\ldots,$ $F^g_{ n+1}$ do not have common zeroes in the toric variety $X_{P+Q}$ over the algebraic closure of $\K.$
From Theorem \ref{g1} we deduce straightforwardly that $Kos(S'_\K,F^g_1,\ldots, F^g_{n+1})_{\alpha_\delta+\alpha_Q+\alpha_P}$ is exact as a complex of $\K$-vector spaces, which is the first part of the statement.

The last part of the claim also follows straightforwardly from Theorem \ref{g1} because if $\Res_{\cA_1,\ldots, \cA_{n+1}}\neq1,$ then we have that $$\Res_{\cA_1,\ldots, \cA_{n+1}}(f_1,\ldots, f_{n+1})\neq0 \iff V_{X_{P+Q}}(F'_1,\ldots, F'_{n+1})=\emptyset$$ thanks to Proposition \ref{ress}.

The remaining parts of the statement are proven in Propositions \ref{gcd} and \ref{rmt} below. We start by comparing the Koszul complex of the theorem with the complex $Kos(S'_\K,F^g_1,\ldots, F^g_{n})_{\alpha_\delta+\alpha_Q+\alpha_P},$ where we have removed $F^g_{n+1},$ and use the results proven in \cite{BT22} for these kind of complexes.

Bernstein Theorem \cite{ber75} states that  $V_{X_{P+Q}}(F^g_1,\ldots, F^g_n)$ is finite, lies in $\T_{\overline{\mathbb{K}}}^n,$  and its cardinality is equal to  $MV(P_1,\ldots, P_n).$

As the degree $\alpha_\delta+\alpha_Q+\alpha_P$ satisfies the hypothesis
of Theorem $4.3$ in \cite{BT22} (thanks to Propositions \ref{bpf} and  \ref{must} with $k=n$), it turns out that  the following complex ((4.4) in \cite{BT22}) is exact:
\begin{equation}\label{cx4}
\begin{array}{r}
0\to{(S'_\K)}_{\alpha_\delta+\alpha_Q+\alpha_{P_{n+1}}}\stackrel{\tilde{\Psi}_{n}}{\to}\ldots\to\bigoplus\limits_{i=1}^n (S'_\K)_{\alpha_\delta+\alpha_Q+\alpha_P-\alpha_{P_i}}\stackrel{\tilde{\Psi}_{1}}{\to} (S'_\K)_{\alpha_\delta+\alpha_Q+\alpha_P} \quad\\
\stackrel{\tilde{\Psi}_{0}}{\to}  \big(S'_\K/I\big)_{\alpha_\delta+\alpha_Q+\alpha_P}\to0,
\end{array}
\end{equation}
with $I:=\langle F^g_1,\ldots, F^g_n\rangle\subset S'_\K.$ Note that this complex coincides with the complex $Kos(S'_\K,F^g_1,\ldots, F^g_{n})_{\alpha_\delta+\alpha_Q+\alpha_P}$ except for the last term.

From Theorem $4.3$ in \cite{BT22}, we also get that
\begin{equation}\label{mv}
\dim\left(S'_\K/I\right)_{\alpha_\delta+\alpha_Q+\alpha_P}=MV(P_1,\ldots, P_n).
\end{equation} because of Bernstein Theorem.  Note that if $MV(P_1,\ldots, P_n)=0,$  \eqref{cx4}  implies that $Kos(S'_\K,F^g_1,\ldots, F^g_{n})_{\alpha_\delta+\alpha_Q+\alpha_P}$  is exact.

We want to relate \eqref{cx4} with the exact complex $Kos(S'_\K,F^g_1,\ldots, F^g_{n}, F^g_{n+1})_{\alpha_\delta+\alpha_Q+\alpha_P},$ which is
\begin{equation}\label{cx5}
0\to{(S'_\K)}_{\alpha_\delta+\alpha_Q}\stackrel{\Psi_{n+1}}{\to} \ldots\to\bigoplus_{i=1}^{n+1} (S'_\K)_{\alpha_\delta+\alpha_Q+\alpha_P-\alpha_{P_i}}\stackrel{\Psi_{1}}{\to} (S'_\K)_{\alpha_\delta+\alpha_Q+\alpha_P}\to0.
\end{equation}
In \eqref{cx4}, all the $\K$-vector spaces have standard monomial bases, except the last one. Clearly one can choose a (class of) monomial basis in this quotient, which we also assume that has been fixed. Let $\widetilde{\Det}$ be the determinant of \eqref{cx4} with respect to the monomial bases of all the nontrivial vector spaces, and $\Det$ the determinant of \eqref{cx5} with respect to the monomial bases of its nontrivial vector spaces.  Following Proposition \ref{alt}, let $\tilde{M}_i,\, 0\leq i\leq n,$ be a square submatrix of the one given by $\tilde{\Psi}_{i}$ such that $\det(\tilde{M}_i)\neq0$ for all $i,$ so that
$$\widetilde{\Det}=\pm\prod_{i=0}^n \det(\tilde{M}_i)^{(-1)^{i}}.
$$
Note that obviously $\tilde{M}_i$ does not depend on the coefficients of $F^g_{n+1}$ for all $i=0,\ldots,n.$
\begin{lemma}\label{aucs}
There exists a square matrix $M_1$ of maximal rank of the form $\begin{pmatrix}\tilde{M}_1& A\\
* &\,B\end{pmatrix},$ such that it is a maximal square submatrix of the matrix of $\Psi_{1},$ with $\begin{pmatrix} A\\B\end{pmatrix}$ having $MV(P_1,\ldots, P_n)$ columns containing each of them coefficients of $F^g_{n+1}.$
\end{lemma}
\begin{proof}
From \eqref{cx4} and \eqref{mv}, we deduce that the image of $\tilde{\Psi}_{1}$ has corank $MV(P_1,\ldots, P_n)$ in $(S'_\K)_{\alpha_\delta+\alpha_Q+\alpha_P}.$ As \eqref{cx5} is exact, we deduce that one can complete the linearly independent set of columns of  the submatrix of $\Psi_{1}$  indexed by the columns of $\tilde{M}_1$ (which has maximal rank),  with $MV(P_1,\ldots, P_n)$ columns  coming from the multiplication by $F^g_{n+1}$. This concludes with the proof.
\end{proof}

\begin{proposition}\label{lab}
Given $M_1$ as above, for all $i=2,\ldots, n$, after sorting properly the monomial bases, there exists a submatrix of maximal rank of $\Psi_{i}$ of the form
$$M_i=\begin{pmatrix}
\tilde{M}_i& *\\
{\bf0}& L_i
\end{pmatrix},
$$
with $L_i$ not depending on the coefficients of $F^g_{n+1},$ and  there is another matrix $M_{n+1}$ -submatrix of maximal rank of $\Psi_{n+1}$ not depending on the coefficients of $F^g_{n+1}$- such that
\begin{equation}\label{det}
\Det =\pm\prod_{i=1}^{n+1} \det(M_i)^{(-1)^{i+1}}.
\end{equation}
\end{proposition}

\begin{proof}
For $i=2, \dots, n$, we can decompose $\Psi_{i}$ as follows

\begin{equation}\label{dec}
\begin{array}{ccl}\bigg(\bigoplus\limits_{1\leq j_1<\ldots<j_i\leq n}(S'_\K)_{\alpha_\delta+\alpha_Q+\alpha_P-\sum\limits_{s=1}^i \alpha_{P_{j_s}}}\bigg) & \oplus &\bigg(\bigoplus\limits_{1\leq j_1<\ldots<j_{i-1}\leq n}(S'_\K)_{\alpha_\delta+\alpha_Q+\alpha_P-\alpha_{P_{n+1}}-\sum\limits_{s=1}^{i-1} \alpha_{P_{j_s}}}\bigg)
\\
& \downarrow &
\\
\bigg(\bigoplus\limits_{1\leq j_1<\ldots<j_{i-1}\leq n}(S'_\K)_{\alpha_\delta+\alpha_Q+\alpha_P-\sum\limits_{s=1}^{i-1} \alpha_{P_{j_s}}}\bigg)&\oplus&\bigg(\bigoplus\limits_{1\leq j_1<\ldots<j_{i-2}\leq n}(S'_\K)_{\alpha_\delta+\alpha_Q+\alpha_P-\alpha_{P_{n+1}}-\sum\limits_{s=1}^{i-2} \alpha_{P_{j_s}}}\bigg)
\end{array}
\end{equation}
(when $i=2$, the last factor has only one nontrivial term $(S'_\K)_{\alpha_\delta+\alpha_Q+\alpha_P-\alpha_{P_{n+1}}}$).

Let $B_i = \tilde{B}_i \cup B'_i$, for $i=1,\dots, n$, be the decomposition of the chosen basis of $\bigoplus\limits_{1\leq j_1<\ldots<j_i\leq n+1}(S'_\K)_{\alpha_\delta+\alpha_Q+\alpha_P-\sum\limits_{s=1}^i \alpha_{P_{j_s}}}$, where $\tilde{B}_i$ is a basis of $\bigoplus\limits_{1\leq j_1<\ldots<j_i\leq n}(S'_\K)_{\alpha_\delta+\alpha_Q+\alpha_P-\sum\limits_{s=1}^i \alpha_{P_{j_s}}}$ and $B_i'$ is a basis of $\bigoplus\limits_{1\leq j_1<\ldots<j_{i-1}\leq n}(S'_\K)_{\alpha_\delta+\alpha_Q+\alpha_P-\alpha_{P_{n+1}}-\sum\limits_{s=1}^{i-1} \alpha_{P_{j_s}}}$.

With this order of bases, we have that the matrix of  $\Psi_{i}$ with respect to the bases $B_i$ and $B_{i-1}$ can be decomposed as
$$|\Psi_i|_{B_i, B_{i-1}}= \begin{pmatrix}
 |\tilde{\Psi}_{i}|_{ \tilde{B}_i, \tilde{B}_{i-1}}& A_i\\
 {\bf 0}& N_i
\end{pmatrix},
$$
with $N_i$ not depending on the coefficients of $F^g_{n+1}.$

Denote with $\tilde{I}_i\subseteq \tilde{B}_i$ the set such that
$\tilde{M}_i = |\tilde{\Psi}_{i}|_{ \tilde{B}_i\setminus \tilde{I}_i, \tilde{I}_{i-1}}$. Here and in the rest of the proof, for a matrix $M$ with columns and rows indexed by sets $\mathcal{C}$ and $\mathcal{R}$ respectively, given $\mathcal{C}_0\subseteq \mathcal{C}$ and $\mathcal{R}_0\subseteq \mathcal{R}$, we will write $M_{\mathcal{C}_0,\mathcal{R}_0}$ for the submatrix of $M$ consisting of the columns indexed by $\mathcal{C}_0$ and the rows indexed by $\mathcal{R}_0$.

Assume that, for every $j<i$, a subset $I_{j}\subseteq B_j$ of the form  $I_{j}=\tilde{I}_{j} \cup I'_{j}$ with  $I'_{j}\subseteq B'_{j}$ has been chosen such that
$M_j= |\Psi_j|_{B_j\setminus I_j, I_{j-1}}$ is a submatrix of maximal rank of $|\Psi_j|_{B_j, B_{j-1}}$ as in the statement of the Proposition.

We have that $\tilde{M}_i$ is a submatrix of $|\tilde{\Psi}_{i}|_{ \tilde{B}_i, \tilde{I}_{i-1}}$ of maximal rank equal to $\#\tilde{I}_{i-1} = \dim(\ker(\tilde{\Psi}_{i-1})) = \dim(\mbox{Im}(\tilde{\Psi}_i))$.
In order to obtain the matrix $M_i$, we consider $|\Psi_i|_{B_i, I_{i-1}}=\begin{pmatrix}
 |\tilde{\Psi}_{i}|_{ \tilde{B}_i, \tilde{I}_{i-1}}& (A_i)_{B'_i,\tilde{I}_{i-1}}\\
 {\bf 0}& (N_i)_{B'_{i},I'_{i-1}}\end{pmatrix}$.
This matrix has full row rank equal to $\#I_{i-1}=\dim(\ker(\Psi_{i-1}))= \dim(\mbox{Im}(\Psi_i))$; then, we can choose $\#I_{i-1}$ linearly independent columns. By the construction of $\tilde{M}_i$, from the columns indexed by $\tilde{B}_i$, those corresponding to $\tilde{B}_i\setminus \tilde{I}_i$ are a maximal linearly independent subset. We extend it to a basis of the column space of $|\Psi_i|_{B_i, I_{i-1}}$  by adding some columns indexed by $B'_i$. Let $I'_i\subseteq B'_i$ be the index set of the columns we do not add and $I_i:= \tilde{I}_i \cup I'_i$. We have that $M_i:=|\Psi_{i}|_{ B_i\setminus I_i, I_{i-1}}$ is a submatrix of $|\Psi_i|_{B_i, B_{i-1}}$ of maximal rank, and it is of the form
$$M_i = \begin{pmatrix}
 |\tilde{\Psi}_{i}|_{ \tilde{B}_i\setminus \tilde{I}_i, \tilde{I}_{i-1}}& (A_i)_{B'_i\setminus I'_i,\tilde{I}_{i-1}}\\
 {\bf 0}& (N_i)_{B'_{i}\setminus I'_i,I'_{i-1}}\end{pmatrix} = \begin{pmatrix}
\tilde{M}_{i} &*\\
 {\bf 0}& L_i\end{pmatrix}$$
The claim now follows straightforwardly since $L_i$ is a submatrix of $N_i$, which does not depend on the coefficients of $F_{n+1}^g$.

To finish the proof, note that having chosen $M_1,\ldots, M_n$ univocally determines $M_{n+1}=|\Psi_{n+1}|_{B_{n+1}, I_{n}}$, since $\Psi_{n+1}$ is injective. As $\tilde{M}_n = |\tilde{\Psi}_n|_{\tilde{B}_n, \tilde{I}_{n-1}}$ (because $\tilde{\Psi}_n$ is also injective),  it turns out that $\tilde{I}_n = \emptyset$ and so, the columns of $M_n$ are indexed by a set $B_n\setminus I_n$ with
$I_n\subseteq B'_n$. Then, the rows of $|\Psi_{n+1}|_{B_{n+1}, I_{n}}$ do not depend on the coefficients of $F_{n+1}^g$. This concludes with the proof of the Proposition.
\end{proof}

\begin{corollary}\label{grad}
For all $i=1,\ldots, n+1,$ the degree of $\Det$ with respect to the coefficients of $F^g_i$ is equal to $MV(P_1,\ldots, \check{P_i},\ldots, P_{n+1}).$
\end{corollary}
\begin{proof}
From Lemma \ref{aucs} and Proposition \ref{lab} we deduce the claim for $i=n+1.$ By permuting the input polynomials, one can get the same statement for any $i=1,\ldots, n.$ \end{proof}

\begin{proposition}\label{gcd}
Up to an invertible element in $\Q,\, \Det$ is equal to the $\gcd$ in $\Q[c_{i,b}]$ of all the maximal minors of $\Psi_{1}.$
\end{proposition}

\begin{proof}
Let $R:=\Q[c_{i,b}],$ and consider $Kos(S'_R,F^g_1,\ldots, F^g_{n+1})_{\alpha_\delta+\alpha_Q+\alpha_P},$
which is a complex of $R$-modules. As this complex becomes  $Kos(S'_\K,F^g_1,\ldots, F^g_{n+1})_{\alpha_\delta+\alpha_Q+\alpha_P}$ after tensoring it with $K(R)=\K,$ which is exact, we deduce then that it is generically exact as defined in \eqref{cfx}. We will use Theorem \ref{mult} to prove our claim.

To do so, let $p\in R$ be any irreducible polynomial. Then, $p$ depends on some coefficient $c_{i_0,b}$. Assume w.l.o.g.~that $i_0=n+1.$ From Proposition \ref{lab}, all the matrices $M_j$ are invertible in the local ring $R_{\langle p\rangle},
j\geq 2,$ as these minors cannot be a multiple of $p$ because they do not depend on the coefficients of $F^g_{n+1}.$
We deduce then straightforwardly that $Kos(S'_{R_{\langle p\rangle}},F^g_1,\ldots, F^g_{n+1})_{\alpha_\delta+\alpha_Q+\alpha_P}$
is exact except at the rightmost nontrivial map, which implies then that  the $j$-th cohomology of the complex
$H^j\big(Kos(S'_R,F^g_1,\ldots, F^g_{n+1})_{\alpha_\delta+\alpha_Q+\alpha_P}\big)$ is zero when localized in $\langle p\rangle$ for all $j>0,$ and for all irreducible $p\in R.$ The claim now follows from Theorem \ref{mult}.
\end{proof}

\begin{proposition}\label{rmt}
The $\gcd$ in $\Z[c_{i,b}]$ of all the maximal minors of $\Psi_{1}$ is equal to  $\pm \Res_{\cA_1,\ldots, \cA_{n+1}}.$
\end{proposition}
\begin{proof}
From Proposition \ref{gcd}, we have that $\Det\in\Q[c_{i,b}].$ Consider its factorization in irreducibles in this ring: $\Det=\prod_{j=1}^N p_j^{e_j}.$ If $p_j$ does not coincide up to a nonzero rational factor with $\Elim_{\cA_1,\ldots, \cA_{n+1}},$ then the set
$\{p_j=0\}\cap\{\Elim_{\cA_1,\ldots, \cA_{n+1}}\neq0\}$ is not empty in any algebraically closed field of characteristic zero $K$. Pick a system of polynomials $F_1,\ldots, F_{n+1}\in S_K$ with coefficients in this set. From Theorem \ref{g1}, the specialized complex
$Kos(S'_K,F'_1,\ldots,  F'_{n+1})_{\alpha_\delta+\alpha_Q+\alpha_P}$ is exact. But this implies that the map $\Psi_{1}$ is onto, so one of the maximal minors of it must be nonzero. As $p_j$ is a common factor of all these minors, this gives a contradiction.

So, we have that  $\Det=\lambda\,\Elim_{\cA_1,\ldots, \cA_{n+1}}^e$ for a suitable $e\in\N,$ and $\lambda\in\Q^\times.$ To compute $e$ we use Corollary \ref{grad} and Proposition \ref{ress}(1). We then have that $\Det=\lambda\,\Res_{\cA_1,\ldots, \cA_{n+1}}.$  On the other hand, all the maximal minors of $\Psi_{1}$ belong to $\Z[c_{i,b}],$ and we have just shown that there are no irreducible elements in $\Q[c_{i,b}]$ dividing it apart from $\Elim_{\cA_1,\ldots, \cA_{n+1}}.$ So, the $\gcd$ of all these maximal minors is equal to $z\in\Z^\times$ times  $\Res_{\cA_1,\ldots, \cA_{n+1}}.$ To show that $z=\pm1,$ it is enough to exhibit a minor which is not zero modulo $p$ for any prime integer $p\in\Z.$  These can be found from the Canny-Emiris type matrices produced in \cite{DJS23}. For instance, in Proposition 4.13 in loc.~cit it is shown that there exists a nonzero maximal minor of $\Psi_{1}$ whose initial term with respect to some monomial order equals a monomial, i.e.~this term does not vanish modulo $p$ for any prime $p\in\Z.$ This completes the proof.
\end{proof}

\end{document}